\documentclass{amsart}

\usepackage{amsmath,amsthm,amssymb}
\usepackage{enumerate}
\usepackage{graphicx}
\usepackage{float}
\usepackage[margin=1.5in]{geometry}
\numberwithin{equation}{section}

\theoremstyle{plain}
\newtheorem{thm}{Theorem}[section]

\newtheorem{lem}[thm]{Lemma}
\newtheorem{prop}[thm]{Proposition}

\newtheorem{ques}[thm]{Question}
\theoremstyle{definition}
\newtheorem{defn}[thm]{Definition}

\theoremstyle{remark}
\newtheorem{rem}[thm]{Remark}
\newtheorem{ex}[thm]{Example}

\usepackage{tikz}
\usetikzlibrary{arrows,shapes,trees}
\usepackage{graphicx}
\usepackage{hyperref}
\usepackage{amssymb}
\usepackage{amsfonts}
\usepackage{amsthm}
\usepackage{amsmath}
\usepackage{epstopdf}
\usepackage{mathrsfs}

\newcommand{\N}{\mathbb N}

\newcommand{\R}{\mathbb R}

\newcommand{\e}{\varepsilon}

\newcommand{\dist}{\mathrm{dist}}
\newcommand{\diam}{\mathrm{diam}}

\renewcommand{\t}{\tilde}
\renewcommand{\l}{\ell}
\newcommand{\la}{\lambda}

\newcommand{\p}{\partial}

\def\XXint#1#2#3{{\setbox0=\hbox{$#1{#2#3}{\int}$ }
\vcenter{\hbox{$#2#3$ }}\kern-.6\wd0}}

\title{Uniformization, $\p$-biLipschitz maps, sphericalization, and inversion}
\author{Clark Butler}

\begin{document}
\begin{abstract}
We define $\p$-biLipschitz homeomorphisms between uniform metric spaces and show that these maps are always quasim\"obius. We also show that a homeomorphism being $\p$-biLipschitz is equivalent to the map being biLipschitz in the quasihyperbolic metrics on these spaces. The proofs of these claims require us to uniformize the quasihyperbolic metric. We further show that all admissible uniformizations of a Gromov hyperbolic space are quasim\"obius to one another by the identity map, including those uniformizations that are based at a point of the Gromov boundary. Using the main results we then show that the sphericalization and inversion operations are compatible with uniformization of hyperbolic spaces in a natural sense.  
\end{abstract}

\maketitle

\section{Introduction}
Our work in this paper has two principal goals. The primary goal is to connect together the bounded uniformizations of Gromov hyperbolic spaces defined by Bonk, Heinonen, and Koskela \cite{BHK} and the unbounded uniformizations of Gromov hyperbolic spaces defined by the author \cite{Bu20}. We will show that these uniformizations are related by a quasim\"obius homeomorphism,  much like how the unit disk and upper half plane in $\R^{2}$ are related by a M\"obius transformation. The secondary goal is to provide a useful local criterion for a homeomorphism between uniform metric spaces to be quasim\"obius. This criterion follow from classical results in quasiconformal mapping theory in the Euclidean case, and has been generalized by V\"ais\"al\"a to the case of uniform domains in Banach spaces \cite{V99}. We take inspiration primarily from \cite{V99} in formulating this condition. In the process of proving the main results we establish several new results on unbounded uniform metric spaces that may be of independent interest. 

We will state our results regarding the secondary goal first. Before proceeding further we will fix some notation for the rest of the paper. In general when we have positive functions $f,g: X \rightarrow (0,\infty)$ defined on a set $X$ we will write $f \asymp_{C} g$ for a constant $C \geq 1$ if the inequality $C^{-1}g \leq f \leq Cg$ holds on $X$. For a metric space $(\Omega,d)$ and $x \in \Omega$ we write $B_{d}(x,r) = \{y \in \Omega:d(x,y) < r\}$ for the open ball of radius $r$ centered at $x$. For a subset $E \subset \Omega$ we write
\[
\dist(x,E) = \inf_{y \in E} d(x,y),
\]  
for the distance from a point to that set. When $\Omega$ is incomplete we write $\bar{\Omega}$ for the completion of $\Omega$ and $\p \Omega = \bar{\Omega} \backslash \Omega$ for the complement of $\Omega$ in $\bar{\Omega}$, which we will refer to as the \emph{metric boundary} of $\Omega$. We will continue to write $d$ for the canonical extension of the metric $d$ on $\Omega$ to the completion $\bar{\Omega}$. For $x \in \Omega$ we write  $d_{\Omega}(x) = \dist(x,\p \Omega)$ for the distance of $x$ to the metric boundary of $\Omega$.   

The key condition that we will consider involves a local Lipschitz-type property in which we rescale distances in the domain and range by distance to the metric boundary before comparing them. A similar condition has previously been considered by V\"ais\"al\"a \cite[Definition 7.2]{V99}. 

\begin{defn}\label{def:controlled}
Let $(\Omega,d)$ and $(\Omega',d')$ be incomplete metric spaces. For a constants $L \geq 0$ and $0 < \la < 1$, a map $f:(\Omega,d) \rightarrow (\Omega',d')$ is \emph{$\p$-Lipschitz} with data $(L,\la)$ if for each $x \in \Omega$ we have for all $y,z \in B_{d}(x,\la d_{\Omega}(x))$, 
\begin{equation}\label{controlled inequality}
\frac{d'(f(y),f(z))}{d_{\Omega'}'(f(x))} \leq  L\frac{d(y,z)}{d_{\Omega}(x)}.
\end{equation}
We say that $f$ is \emph{$\p$-biLipschitz} with data $(L,\la)$ if $f$ is a homeomorphism, $L \geq 1$, and both $f$ and $f^{-1}$ are $\p$-Lipschitz with data $(L,\la)$.
\end{defn}

We note that if $f$ is $\p$-biLipschitz with data $(L,\la)$ then applying \eqref{controlled inequality} for $f$ followed by $f^{-1}$ yields that for $x \in \Omega$ we have for $y$ and $z$ belonging to the smaller ball $B_{d}(x,L^{-1}\la d_{\Omega}(x))$, 
\begin{equation}\label{controlled comparison}
\frac{d'(f(y),f(z))}{d_{\Omega'}'(f(x))} \asymp_{L} \frac{d(y,z)}{d_{\Omega}(x)}.
\end{equation}

We will consider these conditions in the context of uniform metric spaces, which we now define. We remark that, unlike in our previous work \cite{Bu20}, in this paper we will be requiring that uniform metric spaces are locally compact. For a metric space $(\Omega,d)$ and a curve $\gamma: I \rightarrow \Omega$ we write $\l(\gamma)$ for the length of $\gamma$; if $\l(\gamma) < \infty$ then we say that $\gamma$ is \emph{rectifiable}. For an interval $I \subset \R$ and $t \in I$ we write $I_{\leq t} = \{s \in I: s\leq t\}$ and $I_{\geq t} = \{s \in I: s\geq t\}$. For a rectifiable curve $\gamma:I \rightarrow \Omega$ we write $\gamma_{-},\gamma_{+} \in \bar{\Omega}$ for the endpoints of $\gamma$; writing $t_{-} \in [-\infty,\infty)$ and $t_{+} \in (-\infty,\infty]$ for the endpoints of $I$, these are defined by the limits $\gamma(t_{-}) = \lim_{t \rightarrow t_{-}} \gamma(t)$ and $\gamma(t_{+}) = \lim_{t \rightarrow t_{+}} \gamma(t)$ in $\bar{\Omega}$ which exist because $\l(\gamma) < \infty$. 

\begin{defn}\label{def:uniform}For a constant $A \geq 1$ and an interval $I \subset \R$, a rectifiable curve $\gamma: I \rightarrow \Omega$ is \emph{$A$-uniform} if 
\begin{equation}\label{uniform one}
\l(\gamma) \leq Ad(\gamma_{-},\gamma_{+}),
\end{equation}
and if for every $t \in I$ we have
\begin{equation}\label{uniform two}
\min\{\l(\gamma|_{I_{\leq t}}),\l(\gamma|_{I_{\geq t}})\} \leq A d_{\Omega}(\gamma(t)). 
\end{equation}
We say that the metric space $\Omega$ is \emph{$A$-uniform} if $\Omega$ is locally compact and any two points in $\Omega$ can be joined by an $A$-uniform curve. 
\end{defn}

We extend Definition \eqref{def:uniform} to the case of non-rectifiable curves $\gamma: I \rightarrow \Omega$ by replacing \eqref{uniform one} with the condition that $d(\gamma(s),\gamma(t)) \rightarrow \infty$ as $s \rightarrow t_{-}$ and $t \rightarrow t_{+}$. We keep the requirement \eqref{uniform two} the same. Observe that with this extended definition the inequality \eqref{uniform two} implies that an $A$-uniform curve $\gamma$ is always \emph{locally rectifiable}, meaning that each compact subcurve of $\gamma$ is rectifiable. We note that it is easily verified from the definitions that the property of a curve $\gamma$ being $A$-uniform is independent of the choice of parametrization of $\gamma$.

To state our first theorem we will also need to define quasim\"obius homeomorphisms. Quasim\"obius homeomorphisms were first introduced by V\"ais\"al\"a in \cite{V85}; we refer the reader there for a detailed treatment of these maps.  The \emph{cross-ratio} of a quadruple of distinct points $x,y,z,w$ in a metric space $(\Omega,d)$ is defined by
\begin{equation}\label{cross-ratio}
[x,y,z,w]_{d} = \frac{d(x,z)d(y,w)}{d(x,y)d(z,w)}.
\end{equation}

\begin{defn}\label{def:quasimobius}
A homeomorphism $f: (\Omega,d) \rightarrow (\Omega',d')$ between metric spaces is \emph{$\eta$-quasim\"obius} for a homeomorphism $\eta:[0,\infty) \rightarrow [0,\infty)$ if for each quadruple of distinct points $x,y,z,w \in \Omega$ we have
\begin{equation}\label{quasimobius inequality}
[f(x),f(y),f(z),f(w)]_{d'} \leq \eta([x,y,z,w]_{d}).
\end{equation}
\end{defn}

We will commonly refer to $\eta$ as a \emph{control function}. When the control function $\eta$ does not need to be mentioned we simply say that $f$ is \emph{quasim\"obius}. It is easy to see via the symmetries of the cross-ratio that inequality \eqref{quasimobius inequality} implies the two-sided inequality
\begin{equation}\label{two-sided quasimobius inequality}
\frac{1}{\eta([x,y,z,w]_{d}^{-1})} \leq [f(x),f(y),f(z),f(w)]_{d'} \leq \eta([x,y,z,w]_{d}).
\end{equation}

We can now state our first theorem.

\begin{thm}\label{homothety to mobius}
Let $f: (\Omega,d) \rightarrow (\Omega',d')$ be a homeomorphism between two $A$-uniform metric spaces that is $\p$-biLipschitz with data $(L,\la)$. Then $f$ is $\eta$-quasim\"obius with $\eta$ depending only on $A$ and $L$.  
\end{thm} 

Theorem \ref{homothety to mobius} is useful because the $\p$-biLipschitz condition is a local condition that is relatively straightforward to verify in many cases. However it is also somewhat difficult to work with in proofs. Our next theorem (which is used in the proof of Theorem \ref{homothety to mobius}) transforms the $\p$-biLipschitz condition into another condition that is easier to work with. 

We will say that a metric space is \emph{rectifiably connected} if any two points $x,y \in \Omega$ can be joined by a rectifiable curve.  For an incomplete, rectifiably connected metric space $(\Omega,d)$ we define the \emph{quasihyperbolic metric} on $\Omega$ by, for $x,y \in \Omega$, 
\begin{equation}\label{quasihyperbolic metric}
k(x,y) = \inf \int_{\gamma} \frac{ds}{d_{\Omega}(\gamma(s))},
\end{equation}
where the infimum is taken over all rectifiable curves joining $x$ to $y$. The metric space $(\Omega,k)$ is then called the \emph{quasihyperbolization} of the metric space $(\Omega,d)$. The quasihyperbolic metric has a rich history of use in problems in analysis and geometry when $(\Omega,d)$ is a domain in Euclidean space equipped with either the Euclidean metric or an internal path metric; we refer to \cite[Chapter 1]{BHK} for surveys of some of these applications as well as an overview of how they connect to the topics considered in this paper.

For a uniform metric space $(\Omega,d)$ the corresponding quasihyperbolization $(\Omega,k)$ is always a Gromov hyperbolic space \cite[Theorem 3.6]{BHK}. Gromov hyperbolic spaces will be formally defined later in this introduction; for now we observe that transferring problems to the setting of Gromov hyperbolic spaces unlocks a number of additional tools that can be used. For the next theorem we recall that a map $f:(\Omega,d) \rightarrow (\Omega',d')$ between metric spaces is \emph{$H$-Lipschitz} for a constant $H \geq 0$ if for all $x,y \in \Omega$ we have $d'(f(x),f(y)) \leq H d(x,y)$. The map $f$ is \emph{$H$-biLipschitz} if $d'(f(x),f(y)) \asymp_{H} d(x,y)$.

\begin{thm}\label{transfer lip}
Let $f: (\Omega,d) \rightarrow (\Omega',d')$ be a continuous map between two $A$-uniform metric spaces that is $\p$-Lipschitz with data $(L,\la)$. Then there is a constant $H = H(A,L)$ such that the induced map $f:(\Omega,k) \rightarrow (\Omega',k')$ between their quasihyperbolizations is $H$-Lipschitz. 

Conversely, if $f: (\Omega,d) \rightarrow (\Omega',d')$ is such that the induced map $f:(\Omega,k) \rightarrow (\Omega',k')$ is $H$-Lipschitz for a given $H \geq 0$ then there is $L = L(A,H)$ and $\la = \la(A,H)$ such that $f$ is $\p$-Lipschitz with data $(L,\la)$. 
\end{thm}

Thus for homeomorphisms between uniform metric spaces the $\p$-biLipschitz property is equivalent to the induced map between the quasihyperbolizations being biLipschitz. The proof of the first direction is a short calculation (which shows that we may in fact take $H = 4A^{2}L$), while the proof of the other direction is significantly more involved as it makes use of uniformizations of the quasihyperbolizations of the spaces. Theorems \ref{homothety to mobius} and \ref{transfer lip} were inspired by a collection of similar results established by V\"ais\"al\"a in his study of mappings between uniform domains in Banach spaces \cite[Section 11]{V99}.



We move now to a discussion of our primary theorems, which concern uniformizations of Gromov hyperbolic spaces. We begin by recalling some definitions from our previous work \cite{Bu20}; we refer the reader to there for further details. For a continuous function $\rho: X \rightarrow (0,\infty)$ we write
\[
\l_{\rho}(\gamma) = \int_{\gamma} \rho \, ds,
\]
for the line integral of $\rho$ along $\gamma$. Such a positive continuous function $\rho$ will be called a \emph{density} on $X$.

\begin{defn}\label{conformal factor}
Let $(X,d)$ be a rectifiably connected metric space and let $\rho: X \rightarrow (0,\infty)$ be a density on $X$. The \emph{conformal deformation of $X$ with conformal factor $\rho$} is the metric space $X_{\rho} = (X,d_{\rho})$ with metric
\[
d_{\rho}(x,y) = \inf  \l_{\rho}(\gamma),
\]
with the infimum taken over all curves $\gamma$ joining $x$ to $y$. 

If $X$ is geodesic then we say further that $\rho$ is a \emph{Gehring-Hayman density} (abbreviated as \emph{GH-density}) with constant $M \geq 1$ if for any $x,y \in X$ and any geodesic $\gamma$ joining $x$ to $y$ we have
\begin{equation}\label{first GH}
\l_{\rho}(\gamma) \leq M d_{\rho}(x,y). 
\end{equation}
\end{defn}

The quasihyperbolization $(\Omega,k)$ of an incomplete metric space $(\Omega,d)$ is an example of a conformal deformation with conformal factor $\rho(x) = d_{\Omega}(x)^{-1}$.

A metric space $(X,d)$ is \emph{proper} if its closed balls are compact. A \emph{geodesic} $\gamma: I \rightarrow X$ is a curve $\gamma$ which is isometric as a mapping of the interval $I$ into $X$, or in other words, it satisfies $d(\gamma(s),\gamma(t)) = |s-t|$ for all $s,t \in I$. We say that $X$ is \emph{geodesic} if any two points in $X$ can be joined by a geodesic. A \emph{geodesic triangle} $\Delta$ in $X$ consists of three points $x,y,z \in X$ together with geodesics joining these points to one another. Writing $\Delta = \gamma_{1} \cup \gamma_{2} \cup \gamma_{3}$ as a union of its edges, we say that $\Delta$ is \emph{$\delta$-thin} for a given $\delta \geq 0$ if for each point $p \in \gamma_{i}$, $i =1,2,3$, there is a point $q \in \gamma_{j}$ with $d(p,q) \leq \delta$ and $i \neq j$. A geodesic metric space $X$ is \emph{Gromov hyperbolic} if there is a $\delta \geq 0$ such that all geodesic triangles in $X$ are $\delta$-thin; in this case we will also say that $X$ is \emph{$\delta$-hyperbolic}. 

For a proper geodesic $\delta$-hyperbolic space $X$ its \emph{Gromov boundary} $\p X$ is defined to be the set of all geodesic rays $\gamma:[0,\infty) \rightarrow X$ up to the following equivalence relation: $\gamma \sim \sigma$ if there is a constant $c \geq 0$ such that $d(\gamma(t),\sigma(t)) \leq c$ for all $t \geq 0$. In this case we say that $\gamma$ and $\sigma$ are at \emph{bounded distance} from each other. For a geodesic line $\gamma: \R \rightarrow X$ its endpoints in the Gromov boundary $\p X$ are defined to be the equivalence classes $\xi = [\gamma|_{[0,\infty)}]$ and $\zeta = [\bar{\gamma}|_{[0,\infty)}]$, where we write $\bar{\gamma}(t)= \gamma(-t)$ for $\gamma$ with its reversed orientation. We then say that $\gamma$ starts from $\zeta$ and ends at $\xi$. 

 We note that there is a formal conflict between the notation of $\p X$ for the Gromov boundary and the notation $\p X = \bar{X} \backslash X$ for the metric boundary of $X$. However when $X$ is proper we always have $\bar{X} = X$ so that the metric boundary is always empty. Thus there should be no ambiguity in the use of this notation.  

We define rough starlikeness next. 

\begin{defn}\label{def:rough star}Given $K \geq 0$ we say that $X$ is \emph{$K$-roughly starlike} from a point $z \in X$ if for each $x \in X$ there is a geodesic ray $\gamma: [0,\infty) \rightarrow X$ with $\gamma(0) = z$ and $\dist(x,\gamma) \leq K$. We extend this notion to points of the Gromov boundary by saying that $X$ is $K$-roughly starlike from a point $\omega \in \p X$ if for each $x \in X$ there is a geodesic line $\gamma: \R \rightarrow X$ starting from $\omega$ with $\dist(x,\gamma) \leq K$.
\end{defn}

\begin{rem}\label{different star}
The definition of $K$-rough starlikeness we give here corresponds to condition (1) of $K$-rough starlikeness defined in \cite[Definition 2.3]{Bu20}. For \emph{proper} geodesic Gromov hyperbolic spaces condition (2) of that reference always holds, so these definitions are equivalent for proper spaces. 
\end{rem}

The \emph{Busemann function} $b_{\gamma}: X \rightarrow \R$ associated to a geodesic ray $\gamma:[0,\infty) \rightarrow X$ is defined by the limit
\begin{equation}\label{first busemann definition}
b_{\gamma}(x) = \lim_{t \rightarrow \infty} d(\gamma(t),x)-t. 
\end{equation}
The Busemann function $b_{\gamma}$ is a $1$-Lipschitz function on $X$. We will generally omit $\gamma$ from the notation and write $b = b_{\gamma}$. We write 
\begin{equation}\label{extension busemann definition}
\mathcal{B}(X) = \{b_{\gamma}+s: \text{$\gamma$ a geodesic ray in $X$, $s \in \R$}\},
\end{equation}
for the set of all Busemann functions associated to geodesic rays in $X$ as well as all translates of these functions by additive constants, which we will also refer to as Busemann functions. For a point $\omega \in \p X$ we say that a Busemann function $b \in \mathcal{B}(X)$ is \emph{based at $\omega$} if $b = b_{\gamma}+s$ for some $s \in \R$ and some geodesic ray $\gamma$ belonging to the equivalence class of $\omega$. Conversely, given such a Busemann function $b = b_{\gamma}+s$ we refer to the equivalence class $[\gamma] \in \p X$ as its \emph{basepoint} and write $\omega_{b} = [\gamma]$ for this basepoint. 

For $z \in X$ we define $b_{z}(x) = d(x,z)$ to be the distance from $z$. We augment the set of Busemann functions with the set of translates of distance functions on $X$, 
\begin{equation}\label{distance definition}
\mathcal{D}(X) = \{b_{z}+s :z\in X, s\in \R\}.
\end{equation}
For $b \in \mathcal{D}(X)$ with $b = b_{z}+s$ for some $z \in X$ and $s \in \R$ we then refer to $z$ as the \emph{basepoint} of $b$ and will sometimes write $\omega_{b} = z$ in analogy to the case of Busemann functions. We write $\hat{\mathcal{B}}(X) = \mathcal{D}(X) \cup \mathcal{B}(X)$. Then all functions $b \in \hat{\mathcal{B}}(X)$ are $1$-Lipschitz. For $\e > 0$ and $b \in \hat{\mathcal{B}}(X)$ we define a density $\rho_{\e,b}(x) = e^{-\e b(x)}$ on $X$. We write $X_{\e,b} = X_{\rho_{\e,b}}$ for the conformal deformation of $X$ with conformal factor $\rho_{\e,b}$. In the case $b = b_{z}$ for some $z \in X$ we will also use the notation $\rho_{\e,z}(x) = e^{-\e d(x,z)}$ and write $X_{\e,z}$ for the corresponding  conformal deformation of $X$.

By \cite[Theorem 1.4]{Bu20}, given $b \in \hat{\mathcal{B}}(X)$ such that $X$ is $K$-roughly starlike from the basepoint $\omega_{b}$ of $b$ and given $\e > 0$ such that $\rho_{\e,b}$ is a GH-density on $X$ with constant $M$ then geodesics in $X$ are $A$-uniform curves in $X_{\e,b}$ with $A = A(\delta,K,\e,M)$ depending only on these parameters as well as the hyperbolicity parameter $\delta$. In particular the metric space $X_{\e,b}$ is $A$-uniform. The space $X_{\e,b}$ is bounded when $b \in \mathcal{D}(X)$ and unbounded when $b \in \mathcal{B}(X)$. It is natural to inquire as to what extent the uniformization $X_{\e,b}$ depends on the choice of $\e > 0$ and $b \in \hat{\mathcal{B}}(X)$. This is the topic of our next theorem. 

\begin{thm}\label{quasimobius uniformizations}
Let $X$ be a proper geodesic $\delta$-hyperbolic space. Let $\e,\e' > 0$ and $b,b' \in \hat{\mathcal{B}}(X)$ be given such that $X$ is $K$-roughly starlike from the basepoints of both $b$ and $b'$ and both $\rho_{\e,b}$ and $\rho_{\e',b'}$ are GH-densities with the same constant $M$. Then the homeomorphism $X_{\e,b} \rightarrow X_{\e',b'}$ induced by the identity map on $X$ is $\p$-biLipschitz with data $(L,\la)$ depending only on $\delta$, $K$, $\e$, $\e'$, and $M$. Furthermore this map is $\eta$-quasim\"obius with 
\[
\eta(t) = C\max\{t^{\frac{\e'}{\e}},t^{\frac{\e}{\e'}}\},
\]
where $C = C(\delta,K,\e,\e',M)$. 
\end{thm}

Note in particular that we obtain $\eta(t) = Ct$ in the case $\e = \e'$. Theorem \ref{quasimobius uniformizations} provides sharper control over the form of the control function $\eta$ than would be given by Theorem \ref{homothety to mobius}. The uniformizations $X_{\e,b}$ for $b \in \mathcal{D}(X)$ are a slight generalization of the uniformizations considered by Bonk, Heinonen, and Koskela in their work \cite{BHK}; Theorem \ref{quasimobius uniformizations} thus relates these uniformizations to the uniformizations associated to Busemann functions built by the author in \cite{Bu20}. 

We next consider inversion and sphericalization in the context of uniformization. We recall these notions as defined in the work of Buckley, Herron, and Xie \cite{BHX08}, starting with inversion. For an incomplete metric space $\Omega$ we fix a point $p \in \bar{\Omega}$. We define a function $i^{p}$ on $\Omega$ by 
\begin{equation}\label{define inversion}
i^{p}(x,y) = \frac{d(x,y)}{d(x,p)d(y,p)}.
\end{equation}
This function may not define a metric on $\Omega$, but it is $4$-biLipschitz to a canonically defined metric $d^{p}$ on $\Omega$ by \cite[Lemma 3.2]{BHX08}. The \emph{inversion} of $\Omega$ about the point $p$ is defined to be the metric space $\Omega^{p} = (\Omega,d^{p})$. Since $p \in \p \Omega$, $p$ is not an isolated point of $\bar{\Omega}$ and therefore $\Omega^{p}$ is unbounded by \cite[Lemma 3.2 (a)]{BHX08}.

We now describe the sphericalization construction in \cite{BHX08}. We will consider sphericalization as a special case of inversion, as described at the end of \cite[Section 3.B]{BHX08}. We start with a  metric space $(\Omega,d)$ with a fixed choice of point $p \in \Omega$. We define an auxiliary metric space $\hat{\Omega} = \Omega \cup_{p \sim 0} [0,1]$ by attaching the interval $[0,1]$ to $\Omega$ by identifying $p$ with $0$ and then giving the resulting space the metric $\hat{d}$ which restricts to $d$ on $\Omega$, restricts to the Euclidean metric on $[0,1]$, and for $x \in \Omega$, $t \in [0,1]$ is given by $\hat{d}(x,t) = d(x,p) + t$. We then let $\hat{d}^{p}$ be the metric on $\hat{\Omega} \backslash \{1\}$ defined by taking the inversion of $\hat{\Omega}$ about the point $1 \in [0,1]$ and write $\hat{\Omega}^{p} = (\Omega,\hat{d}^{p})$ for the metric space resulting from restricting this metric to $\Omega$. The metric space  $\hat{\Omega}^{p}$ is called the \emph{sphericalization} of $\Omega$ based at $p$. The space $\hat{\Omega}^{p}$ always satisfies $\diam \, \hat{\Omega}^{p} \leq 1$ by the discussion in \cite[Section 3.B]{BHX08}, and in particular is always bounded. 

Since inversion produces an unbounded space from a bounded space and sphericalization produces a bounded space from an unbounded space, given a $\delta$-hyperbolic space $X$ as in Theorem \ref{quasimobius uniformizations} one may ask whether the bounded uniformizations $X_{\e,b}$ for $b \in \mathcal{D}(X)$ and the unbounded uniformizations $X_{\e,b}$ for $b \in \mathcal{B}(X)$ can be related to one another by the operations of sphericalization and inversion. Using the work of \cite{BHX08} we will show below that this is indeed the case. To state the result properly we require another definition. 

\begin{defn}\label{def:quasisymmetry}
A homeomorphism $f:(\Omega,d) \rightarrow (\Omega',d')$ between metric spaces is \emph{$\eta$-quasisymmetric} for a homeomorphism $\eta: [0,\infty) \rightarrow [0,\infty)$ if for each triple of points $x,y,z \in \Omega$ and each $t \geq 0$ we have
\[
d(x,y) \leq td(x,z) \Rightarrow d'(f(x),f(y)) \leq \eta(t)d'(f(x),f(z)).
\]
\end{defn}

As with quasim\"obius maps, we will call $\eta$ a control function and will simply say that $f$ is quasisymmetric if the control function does not need to be mentioned. Quasisymmetric maps are always quasim\"obius \cite[Theorem 3.2]{V85} (with the control function for the quasim\"obius condition being quantitative in the control function for the quasisymmetry condition), but the reverse need not always be the case since quasisymmetric maps must take bounded sets to bounded sets \cite[Proposition 10.8]{Hein01} while the same is not true of quasim\"obius homeomorphisms. 

In the statement of Theorem \ref{inversion theorem} below we will restrict to $b \in \mathcal{D}(X)$ of the form $b_{z}(x) = d(x,z)$ for some $z \in X$, and we will only consider Busemann functions $b \in \mathcal{B}(X)$ such that $b(z) = 0$. These restrictions are harmless normalizations as any $b \in \mathcal{D}(X)$ or $b \in \mathcal{B}(X)$ will only differ from such a function by an additive constant. In the formulation of Theorem \ref{inversion theorem} we are implicitly using the fact that for $z \in X$ the metric boundary $\p X_{\e,z}$ can be identified with the Gromov boundary $\p X$ in a canonical fashion \cite[Theorem 1.6]{Bu20}. In particular for $b \in \mathcal{B}(X)$ with basepoint $\omega \in \p X$ we can consider $\omega$ as a point of $\p X_{\e,z}$.

\begin{thm}\label{inversion theorem} 
Let $X$ be a proper geodesic $\delta$-hyperbolic space that is $K$-roughly starlike from a point $z \in X$ and a point $\omega \in \p X$. Fix a Busemann function $b$ based at $\omega$ satisfying $b(z) = 0$ and suppose that $\e > 0$ is chosen such that $\rho_{\e,z}$ and $\rho_{\e,b}$ are GH-densities with the same constant $M$. Let $X_{\e,z}^{\omega}$ be the inversion of $X_{\e,z}$ about the point $\omega \in \p X_{\e,z}$ and let $\hat{X}_{\e,b}^{z}$ be the sphericalization of $X_{\e,b}$ based at $z$. 

Then the metric spaces $X_{\e,z}^{\omega}$ and $\hat{X}_{\e,b}^{z}$ are each $A'$-uniform with $A' = A'(\delta,K,\e,M)$ and the maps $X_{\e,b} \rightarrow X_{\e,z}^{\omega}$ and $X_{\e,z} \rightarrow \hat{X}_{\e,b}^{z}$ induced by the identity map on $X$ are each $\p$-biLipschitz with data $(L,\la)$ and $\eta$-quasisymmetric with $L$, $\la$, and $\eta$ depending only on $\delta$, $K$, $\e$, and $M$. 
\end{thm}

The uniformity of the metric spaces $X_{\e,z}^{\omega}$ and $\hat{X}_{\e,b}^{z}$ is a straightforward consequence of the fact shown in \cite{BHX08} that sphericalization and inversion preserve uniformity of metric spaces. Theorem \ref{homothety to mobius} immediately implies from the $\p$-biLipschitz condition that the maps $X_{\e,b} \rightarrow X_{\e,b}^{\omega}$ and $X_{\e,z} \rightarrow \hat{X}_{\e,b}^{z}$ are quasim\"obius, but we are able to obtain with a small amount of additional work that they are actually quasisymmetric with quantitative control on the control function.

\begin{rem}\label{compare quasi} In \cite[Chapter 1]{BHK} the term \emph{quasisimilarity} is used for a notion that is both weaker and stronger than the property of being $\p$-biLipschitz that we define in this paper. They define a homeomorphism $f: (\Omega,d) \rightarrow (\Omega',d')$ of incomplete metric spaces to be a \emph{quasisimilarity} with data $(\eta,L,\la)$ if $f$ is $\eta$-quasisymmetric and for each $x \in \Omega$ there is a constant $c_{x} > 0$ such that whenever $y,z \in B(x,\la d_{\Omega}(x))$ we have 
\begin{equation}\label{quasisimilar inequality}
d'(f(y),f(z)) \asymp_{L} c_{x}d(y,z).
\end{equation}
The $\p$-biLipschitz property corresponds to requiring that $c_{x} = \frac{d_{\Omega'}'(x)}{d_{\Omega}(x)}$ (see the comparison \eqref{controlled comparison}). With this tweak to the definition, Theorem \ref{homothety to mobius} shows in many cases that the comparison \eqref{quasisimilar inequality} is sufficient on its own to deduce the quasisymmetry property, including in the case of maps between bounded uniform metric spaces considered in \cite{BHK} (see (2) of Remark \ref{specializations}). 
\end{rem}

Lastly, it is important to know when the hypotheses of Theorems \ref{quasimobius uniformizations} and \ref{inversion theorem} are actually satisfied. The key hypotheses are the rough starlikeness condition and the fact that $\rho_{\e,b}$ is a GH-density. For rough starlikeness it turns out that rough starlikeness from a single point of $X$ is sufficient to obtain rough starlikeness from \emph{all} points of $X \cup \p X$, provided that $\p X$ contains at least two points. 

\begin{prop}\label{all star}
Let $X$ be a proper geodesic $\delta$-hyperbolic space such that $\p X$ contains at least two points, and suppose that $X$ is $K$-roughly starlike from a point $x \in X \cup \p X$. Then there is a constant $K' \geq 0$ such that $X$ is $K'$-roughly starlike from all points of $X \cup \p X$. 
\end{prop}

This proposition has been obtained independently by Zhou with a similar proof \cite[Lemma 3.3]{Z20}. 

Proposition \ref{all star} shows that we need not worry about the dependence on the basepoint in the rough starlikeness hypothesis. The claim is false when $\p X$ consists of a single point, as for instance the Euclidean half-line $[0,\infty)$ is $0$-hyperbolic and $0$-roughly starlike from the origin $0$, but is at best $t$-roughly starlike from any point $t \in [0,\infty)$ since the only geodesic ray starting from $t$ is the half-line $[t,\infty)$. However, in some sense this is the only counterexample: when $\p X$ consists of a single point and $X$ is roughly starlike from one of its points then $X$ is roughly isometric to the half-line $[0,\infty)$ by Proposition \ref{rough ray} and most claims of interest can be verified by direct argument. The nature of the dependence of $K'$ on the other parameters is summarized in Proposition \ref{rough star boundary}. We remark that $K'$ is not always quantitative solely in $\delta$ and $K$ when $x \in X$, as can be seen in Example \ref{dependence}. A bound for $K'$ can be computed in terms of $\delta$, $K$, and the Gromov product of any two distinct points in $\p X$ using Lemma \ref{two points}.

Regarding the question of when $\rho_{\e,b}$ is a GH-density, we will make use of the following imporant theorem due to Bonk, Heinonen, and Koskela; we note that the version in the reference is somewhat more general. This is known as a \emph{Gehring-Hayman}-type theorem, after the namesakes' corresponding result in the context of Euclidean and hyperbolic metrics on hyperbolic domains in the complex plane \cite{GH62}. 

\begin{thm}\label{Gehring-Hayman}\cite[Theorem 5.1]{BHK}
Let $(X,d)$ be a geodesic $\delta$-hyperbolic space. There is $\e_{0} = \e_{0}(\delta) > 0$ depending only on $\delta$ such that if a continuous function $\rho:X \rightarrow (0,\infty)$ satisfies for all $x,y \in X$ and some fixed $0 < \e \leq \e_{0}$,
\begin{equation}\label{proto Harnack}
e^{-\e d(x,y)} \leq \frac{\rho(x)}{\rho(y)} \leq e^{\e d(x,y)},
\end{equation}
then $\rho$ is a GH-density $X$ with constant $M = 20$. 
\end{thm}

The inequality \eqref{proto Harnack} is known as a \emph{Harnack type inequality}. For a given $\e > 0$ and any $b \in \hat{\mathcal{B}}(X)$ the function $\rho_{\e,b}$ satisfies inequality \eqref{proto Harnack}  since $b$ is $1$-Lipschitz. By Theorem \ref{Gehring-Hayman} the property of being a GH-density is a non-issue once $\e$ is sufficiently small. There are cases in which one must consider values of $\e$ larger than the $\e_{0}$ given above, however. Such values of $\e$ appear naturally when considering CAT$(-1)$ spaces \cite{Bu20} and uniformizations of hyperbolic fillings of metric spaces \cite{BBS21}, \cite{Bu20}. 

Combining Proposition \ref{all star} and Theorem \ref{Gehring-Hayman} we see that, for a given proper geodesic $\delta$-hyperbolic space $X$ (with $\p X$ having at least two points) that is $K$-roughly starlike from some point of $X \cup \p X$, there is always a $K' \geq 0$ and an $\e_{0} = \e_{0}(\delta)$ such that for any $b \in \hat{\mathcal{B}}(X)$ and any $0 < \e \leq \e_{0}$ we have that $X$ is $K'$-roughly starlike from the basepoint of $b$ and the density $\rho_{\e,b}$ is admissible for $X$ with constant $M = 20$. Thus Theorems \ref{quasimobius uniformizations} and \ref{inversion theorem} can be applied freely in this range. 

Proposition \ref{all star} raises the interesting question of whether a similar phenomenon also holds for the densities $\rho_{\e,b}$ for $\e > 0$, $b \in \hat{\mathcal{B}}(X)$. 

\begin{ques}\label{admissibility question}
Suppose that $X$ is a proper geodesic $\delta$-hyperbolic space and suppose that $\e > 0$ and $b \in \hat{\mathcal{B}}(X)$ are such that $\rho_{\e,b}$ is a GH-density with constant $M$. Is there a constant $M'$ such that for all $b' \in  \hat{\mathcal{B}}(X)$ the density $\rho_{\e,b'}$ is a GH-density with constant $M'$?
\end{ques}

Theorem \ref{Gehring-Hayman} shows that there is a universal constant $M$ such that $\rho_{\e,b}$ is a GH-density for any $b \in \hat{\mathcal{B}}(X)$ once $\e$ is sufficiently small. For CAT$(-1)$ spaces \cite[Theorem 1.10]{Bu20} shows that there is a universal constant $M$ such that $\rho_{1,b}$ is a GH-density for any $b \in \hat{\mathcal{B}}(X)$. Question \eqref{admissibility question} asks whether this phenomenon generalizes to any Gromov hyperbolic space $X$.  This question is in a sense complementary to a recent theorem of Lindquist and Shanmugalingam \cite{LS20} that concerns the invariance of the uniformity property of $X_{\e,z}$ for a given $\e > 0$, $z \in X$ under rough isometries of Gromov hyperbolic spaces. Another interesting question to consider is whether their theorem also holds for uniformizations by Busemann functions instead.

In Section \ref{sec:hyperbolic} we review some basic facts about Gromov hyperbolic spaces and recall some results from our previous paper \cite{Bu20}. We also prove Proposition \ref{all star}. In Section \ref{sec:quasihyperbolize} we study the quasihyperbolic metric on unbounded uniform metric spaces and prove several claims that will be required later in the paper. Section \ref{sec:uniformize quasihyperbolic} is devoted to uniformizing the quasihyperbolic metric on a uniform metric space. Section \ref{sec:quasi to uniform} considers what happens if we quasihyperbolize a uniformization of a Gromov hyperbolic space. The results in Section \ref{sec:quasi to uniform} are not required in the rest of the paper. In Section \ref{sec:main theorems} we complete the proofs of the main theorems. 

Lastly we remark that a number of the foundational results in this paper in Sections 3-5 appeared simultaneously and independently in work of Zhou \cite{Z20}. One may also find there a local-to-global theorem similar in spirit to Theorem \ref{homothety to mobius} \cite[Theorem 1.4]{Z20}, but phrased in terms of quasisymmetric maps instead of quasim\"obius maps.

\section{Hyperbolic metric spaces}\label{sec:hyperbolic}

In this section we recall some standard results regarding Gromov hyperbolic spaces, as well as some facts regarding Busemann functions established in our previous paper.  We will also prove Proposition \ref{all star}. The first parts of this section are closely modeled on the corresponding section of our previous paper \cite[Section 2]{Bu20}, but with appropriate simplifications now that we are assuming our Gromov hyperbolic spaces are proper. The material in this section can be found in any standard reference on Gromov hyperbolic spaces, for instance \cite{BS07}, \cite{GdH90}. 

\subsection{Definitions}\label{subsec:defn} Let $X$ be a set and let $f$, $g$ be real-valued functions defined on $X$. For $c \geq 0$ we will write $f \doteq_{c} g$ if
\[
|f(x)-g(x)| \leq c,
\] 
for all $x \in X$. If the exact value of the constant $c$ is not important or implied by context we will often just write $f \doteq g$. The relation $f \doteq g$ will sometimes be referred to as a \emph{rough equality} between $f$ and $g$.  We will generally stick to the convention throughout this paper of using $c \geq 0$ for additive constants and $C \geq 1$ for multiplicative constants. To indicate on what parameters -- such as $\delta$ -- the constants depend on we will write $c = c(\delta)$, etc. 

Let $f:(X,d) \rightarrow (X',d')$ be a map between metric spaces. We say that $f$ is \emph{isometric} if $d'(f(x),f(y)) = d(x,y)$ for $x$, $y\in X$. For a constant $c \geq 0$ we say that $f$ is \emph{$c$-roughly isometric} if $d(f(x),f(y)) \doteq_{c} d(x,y)$ for $x,y \in X$. As usual we will omit the constants in these terms when we don't require them. We recall that a curve $\gamma: I \rightarrow X$ is a \emph{geodesic} if it is an isometric mapping of the interval $I \subset \R$ into $X$. 

When dealing with Gromov hyperbolic spaces $X$ in this paper we will in many cases use the generic distance notation $|xy|:=d(x,y)$ for the distance between $x$ and $y$ in $X$ and the generic notation $xy$ for a geodesic connecting two points $x,y \in X$, even when this geodesic is not unique. We recall that a geodesic triangle $\Delta$ in $X$ is a collection of three points $x,y,z \in X$ together with geodesics $xy$, $xz$, and $yz$ joining these points and serving as the edges of this triangle. We will sometimes alternatively write $xyz = \Delta$ for a geodesic triangle with vertices $x$, $y$ and $z$.

The Gromov boundary $\p X$ of a proper geodesic $\delta$-hyperbolic space $X$ is defined to be the collection of all geodesic rays $\gamma: [0,\infty) \rightarrow X$ up to the equivalence relation of two rays being equivalent if they are at a bounded distance from one another. Using the Arzela-Ascoli theorem it is easy to see in a proper geodesic $\delta$-hyperbolic space that for any points $x,y \in X \cup \p X$ there is a geodesic $\gamma$ joining $x$ to $y$. We will continue to write $xy$ for any such choice of geodesic joining $x$ to $y$. We will allow our geodesic triangles $\Delta$ to have vertices on $\p X$, in which case we will still write $\Delta = xyz$ if $\Delta$ has vertices $x,y,z$. We remark that geodesic triangles with vertices in $X \cup \p X$ are $10\delta$-thin by \cite[Lemma 2.2]{Bu20}. 

We next state a useful fact regarding Busemann functions. Let $X$ be a proper geodesic $\delta$-hyperbolic space and let $b: X \rightarrow \R$ be a Busemann function  based at some point $\omega \in \p X$. By \cite[Lemma 2.5]{Bu20}, if $b'$ is any other Busemann function based at $\omega$ then there is a constant $s \in \R$ such that 
\begin{equation}\label{difference constant}
b \doteq_{72\delta} b'+s,
\end{equation}
with $s = 0$ if the geodesic rays associated to $b$ and $b'$ have the same starting point. Thus all Busemann functions based at $\omega$ differ from each other by an additive constant, up to an additive error of $72\delta$.

\subsection{Gromov products}For $x,y,z \in X$ the \emph{Gromov product} of $x$ and $y$ based at $z$ is defined by
\begin{equation}\label{Gromov product}
(x|y)_{z} = \frac{1}{2}(|xz|+|yz|-|xy|). 
\end{equation} 
We can also take the basepoint of the Gromov product to be any function $b \in \hat{\mathcal{B}}(X)$. For $b \in \hat{\mathcal{B}}(X)$ the Gromov product based at $b$ is defined by 
\begin{equation}\label{Gromov Busemann product}
(x|y)_{b} = \frac{1}{2}(b(x) + b(y) - |xy|). 
\end{equation}
For $b \in \mathcal{D}(X)$, $b(x) = d(x,z)+s$ this reduces to the notion of Gromov product in \eqref{Gromov product}, as we have $(x|y)_{b} = (x|y)_{z} + s$. 

The following statements briefly summarize a more extensive discussion of Gromov products in \cite[Section 2]{Bu20}. In particular we give details there for how the precise forms of these statements follow from the corresponding statements in the literature. For these statements it is useful to conceive of the Gromov boundary in an alternative way using Gromov products. Fix $z \in X$. A sequence $\{x_{n}\} \subset X$ \emph{converges to infinity} if $(x_{m}|x_{n}) \rightarrow \infty$ as $m,n \rightarrow \infty$. Two sequences $\{x_{n}\}$ and $\{y_{n}\}$ are \emph{equivalent} if $(x_{n}|y_{n})_{z} \rightarrow \infty$. These notions do not depend on the choice of basepoint $z$, as can easily be checked by the triangle inequality. For a proper geodesic $\delta$-hyperbolic space $X$ the set of equivalence classes of sequences converging to infinity gives an equivalent definition of the Gromov boundary $\p X$, with the equivalence being given by sending a geodesic ray $\gamma:[0,\infty) \rightarrow X$ to the sequence $\{\gamma(n)\}$. For $\xi \in \p X$ and a sequence $\{x_{n}\}$ that converges to infinity we will $\{x_{n}\} \in \xi$ if $\{x_{n}\}$ belongs to the equivalence class of $\xi$. 

These notions may be extended to Busemann functions $b \in \mathcal{B}(X)$ based at a given point $\omega \in \p X$ \cite[Chapter 3]{BS07}. As above a sequence $\{x_{n}\}$ \emph{converges to infinity with respect to $\omega$} if $(x_{m}|x_{n})_{b} \rightarrow \infty$ as $m,n \rightarrow \infty$, and two sequences $\{x_{n}\}$ and $\{y_{n}\}$ are \emph{equivalent with respect to $\omega$} if $(x_{n}|y_{n})_{b} \rightarrow \infty$ as $n \rightarrow \infty$. These definitions do not depend on the choice of Busemann function based at $\omega$ by \eqref{difference constant}. The \emph{Gromov boundary relative to $\omega$} is defined to be the set $\p_{\omega}X$ of all equivalence classes of sequences converging to infinity with respect to $\omega$. By \cite[Proposition 3.4.1]{BS07} we have a canonical identification of $\p_{\omega}X$ with the complement $\p X \backslash\{\omega\}$ of $\omega$ in the Gromov boundary $\p X$. We will thus use the notation $\p_{\omega}X = \p X \backslash \{\omega\}$ throughout the rest of the paper. We extend this notation to $\omega \in X$ by setting $\p_{\omega}X = \p X$ in this case.

Gromov products based at functions $b \in \hat{\mathcal{B}}(X)$ can be extended to points of $\p X$ by defining the Gromov product of equivalence classes $\xi$, $\zeta \in \p X$ based at $b$ to be
\[
(\xi |\zeta)_{b} = \inf \liminf_{n \rightarrow \infty}(x_{n}|y_{n})_{b},
\] 
with the infimum taken over all sequences $\{x_{n}\} \in \xi$, $\{y_{n}\} \in \zeta$; if $b \in \mathcal{B}(X)$ has basepoint $\omega$ then we leave this expression undefined when $\xi = \zeta = \omega$. As a consequence of \cite[Lemma 2.2.2]{BS07}, \cite[Lemma 3.2.4]{BS07}, and the discussion in \cite[Section 2.2]{Bu20}, for any choices of sequences  $\{x_{n}\} \in  \xi$ and  $\{y_{n}\} \in \zeta$ we have
\begin{equation}\label{sequence approximation}
(\xi |\zeta)_{b}  \leq \liminf_{n \rightarrow \infty}(x_{n}|y_{n})_{b} \leq \limsup_{n \rightarrow \infty}(x_{n}|y_{n})_{b} \leq (\xi |\zeta)_{b} + c(\delta),
\end{equation}
and one has the inequality for any $x,y,z \in X \cup \p X_{\omega}X$ we have
\begin{equation}\label{Gromov inequality}
(x |z)_{b} \geq \min \{(x | y)_{b},(y | z)_{b}\} - c(\delta), 
\end{equation}
with the constant $c(\delta)$ depending only on $\delta$ in each case. One may take $c(\delta) = 8\delta$ for $b \in \mathcal{D}(X)$ and $c(\delta)  = 600\delta$ for $b \in \mathcal{B}(X)$ in \eqref{sequence approximation} and \eqref{Gromov inequality}. For $x \in X$ and $\xi \in \p X$ the Gromov product based at $b$ is defined analogously as 
\begin{equation}\label{extended definition}
(x |\xi)_{b} = \inf \liminf_{n \rightarrow \infty}(x|x_{n})_{b},
\end{equation}
and the analogous inequality \eqref{sequence approximation} holds with the same constants.

For a given $z \in X$ and $b \in \mathcal{D}(X)$ defined by $b(x) = |xz|$, we will also write $(\xi|\zeta)_{z} = (\xi|\zeta)_{b}$ for $\xi,\zeta \in X \cup \p X$. We remark that, with the extended definition \eqref{extended definition}, a sequence $\{x_{n}\}$ belongs to the equivalence class of $\xi \in \p X$ if and only if $(x_{n}|\xi)_{z} \rightarrow \infty$ for some (hence any) $z \in X$.


\subsection{Visual metrics}\label{subsec:visual} Let $X$ be a proper geodesic $\delta$-hyperbolic space. Gromov products based at $b \in \hat{\mathcal{B}}(X)$ can be used to define visual metrics on the Gromov boundary $\p X$. We refer to \cite[Chapters 2-3]{BS07} as well as \cite[Section 2.3]{Bu20} for precise details on this topic. We will summarize the results we need here. For $b \in \hat{\mathcal{B}}(X)$ we let $\omega = \omega_{b}$ denote the basepoint of $b$. We recall that we write $\p_{\omega}X = \p X$ when $\omega \in X$ and $\p_{\omega}X = \p X \backslash \{\omega\}$ when $\omega \in \p X$.

For $b \in \hat{\mathcal{B}}(X)$ and $\e > 0$ we define for $\xi$, $\zeta \in \p_{\omega} X$,
\begin{equation}\label{visual quasi}
\theta_{\e,b}(\xi,\zeta) = e^{-\e (\xi|\zeta)_{b}}.
\end{equation}
This may not define a metric on $\p_{\omega} X$, since the triangle inequality may not hold. However there is always $\e_{*} = \e_{*}(\delta) > 0$ depending only on $\delta$ such that for $0 < \e \leq \e_{*}$ the function $\theta_{\e,b}$ is $4$-biLipschitz to a metric $\theta$ on $\p_{\omega} X$. We refer to $\theta$ as a \emph{visual metric} on $\p_{\omega} X$ based at $b$ and refer to $\e$ as the \emph{parameter} of $\theta$. We give $\p X$ the topology associated to a visual metric based at $b$ for any $b \in \mathcal{D}(X)$. When equipped with a visual metric $\p_{\omega}X$ is a locally compact metric space. It is always a compact metric space when $b \in \mathcal{D}(X)$, or when $b \in \mathcal{B}(X)$ and $\omega$ is an isolated point of $\p X$.  


For any visual metrics $\theta$ and $\theta'$ on $\p X$ based at $b,b' \in \hat{\mathcal{B}}(X)$ with parameters $\e,\e' > 0$ respecitvely, the identity map on $\p X$ induces an $\eta$-quasim\"obius homeomorphism,
\[
(\p X,\theta) \rightarrow (\p X,\theta'),
\]
with $\eta$ of the form $\eta(t) = C\left(\delta,\frac{\e'}{\e}\right)t^{\frac{\e'}{\e}}$ \cite[Corollary 5.2.9]{BS07}; this result is stated in the reference for the special case $\e = \e'$, but the general case can be deduced immediately from the observation that for a metric space $(\Omega,d)$ and any $0 < a \leq 1$ the identity map $(\Omega,d) \rightarrow (\Omega,d^{a})$ is always $\eta$-quasim\"obius with $\eta(t) = t^{a}$. 

\subsection{Tripod maps}

We let $\Upsilon$ be the tripod geodesic metric space composed of three copies $L_{1}$, $L_{2}$, and $L_{3}$ of the closed half-line $[0,\infty)$ identified at $0$.  We denote this identification point by $o$ and will refer to $o$ as the \emph{core} of the tripod $\Upsilon$. The space $\Upsilon$ is $0$-hyperbolic and its Gromov boundary $\p \Upsilon$ consists of three points $\zeta_{i}$, $i = 1,2,3$, corresponding to the half-lines $L_{i}$ thought of as geodesic rays starting from $o$. We write $b_{\Upsilon}:=b_{L_{1}}$ for the Busemann function associated to the half-line $L_{1}$ thought of as a geodesic ray in $\Upsilon$. A quick calculation shows that $b_{\Upsilon}$ is given by $b_{\Upsilon}(s) = -s$ for $s \in L_{1}$ and $b_{\Upsilon}(s) = s$ for $s \in L_{2}$ or $s \in L_{3}$.

Throughout this section we let $X$ be a proper geodesic $\delta$-hyperbolic space and let $\Delta = xyz$ be a geodesic triangle in $X$ with vertices $x,y,z \in X \cup \p X$. Let $(\hat{x},\hat{y},\hat{z})$ be an ordered triple of points with $\hat{x} \in yz$, $\hat{y} \in xy$, $\hat{z} \in xy$, such that $|\hat{x}y| = |\hat{z}y|$, $|\hat{x}z| = |\hat{y}z|$, and $|\hat{z}x| = |\hat{z}y|$ (we allow some of these distances to be infinite, in which case the equalities becomes trivial). The \emph{tripod map} $T: \Delta \rightarrow \Upsilon$ associated to such a triple is the unique $1$-Lipschitz map characterized by the properties that $T(\hat{x}) = T(\hat{y}) = T(\hat{z}) = o$, that $T$ maps $xy$ isometrically into $L_{2} \cup L_{3}$, $xz$ isometrically into $L_{1} \cup L_{3}$, $yz$ isometrically into $L_{1} \cup L_{3}$, and $T(x) \in L_{1}$, $T(y) \in L_{2}$, $T(z) \in L_{3}$. When $x \in \p X$ this final inclusion should instead be understood as $\p T(x) = \zeta_{1}$, recalling that $\zeta_{1} \in \p \Upsilon$ is the point defined by $L_{1}$, and similarly for $y$ and $z$. 

A tripod map \emph{associated to $\Delta$} is by definition a $400\delta$-roughly isometric map $T: \Delta \rightarrow \Upsilon$ that is a tripod map associated to some triple of points in $\Delta$. By \cite[Proposition 3.8]{Bu20} such a tripod map exists for every geodesic triangle in $X$ that has vertices in $X \cup \p X$. Since we will use the existence of this tripod map very frequently throughout this paper, we often will not refer back to this section when we use it. For a tripod map $T: \Delta \rightarrow \Upsilon$ associated to a geodesic triangle $\Delta = xyz$ the \emph{equiradial points} for $T$ are by definition the three points $\hat{x} \in yz$, $\hat{y} \in xz$, $\hat{z} \in xy$ used to construct $T$, which can be equivalently thought of as the preimages of the core $o$, $T^{-1}(o) = \{\hat{x},\hat{y},\hat{z}\}$. 

The following proposition computes Busemann functions $b \in \mathcal{B}(X)$ on geodesic triangles $\Delta$ in $X$ that have the basepoint of $b$ as one of their vertices. In general we will not be giving explicit constants in our claims here since we do not provide explicit constants in our final results. 

\begin{prop}\label{compute Busemann}\cite[Proposition 3.10]{Bu20}
Let $\Delta = \omega xy$ be a geodesic triangle in $X$ with $\omega \in \p X$ and $x,y\in X \cup \p_{\omega} X$. Let $b$ be a Busemann function based at $\omega$. Let $T: \Delta \rightarrow \Upsilon$ be a tripod map associated to $\Delta$. Then
\begin{equation}\label{tripod image}
b(p) \doteq_{c(\delta)} b_{\Upsilon}(T(p)) + (x|y)_{b},
\end{equation}
and therefore
\begin{equation}\label{tripod minimum}
(x|y)_{b} \doteq_{c(\delta)} \inf_{p \in xy} b(p). 
\end{equation}
\end{prop}

The following direct consequence of Proposition \ref{compute Busemann} will be more useful for computing in practice. The notation $(-\infty,a]$ should be interpreted as $(-\infty,a] = \R$ when $a = \infty$.

\begin{prop}\label{compute consequence}
Let $\Delta = \omega x_{1}x_{2}$ be a geodesic triangle in $X$ with $\omega \in \p X$ and $x_{1},x_{2}\in X \cup \p_{\omega} X$. Let $b$ be a Busemann function based at $\omega$. Let $T: \Delta \rightarrow \Upsilon$ be a tripod map associated to $\Delta$. Then there are parametrizations $\gamma_{i}:(-\infty,a_{i}] \rightarrow X$ of $\omega x_{i}$, $a_{i} \in [0,\infty]$, $i = 1,2$, and $\sigma: I \rightarrow X$ of $x_{1}x_{2}$ with $0 \in I$ such that the following properties hold,
\begin{enumerate}
\item The equiradial points of $T$ on $\gamma_{1}$, $\gamma_{2}$, and $\sigma$ are $\gamma_{1}(0)$, $\gamma_{2}(0)$, and $\sigma(0)$ respectively. In particular $\diam\{\gamma_{1}(0),\gamma_{2}(0),\sigma(0)\} \leq c(\delta)$.
\item For $t \leq 0$ we have $|\gamma_{1}(t)\gamma_{2}(t)| \leq c(\delta)$.
\item For $t \in I_{\leq 0}$ we have $|\gamma_{1}(-t)\sigma(t)| \leq c(\delta)$ and for $t \in I_{\geq 0}$ we have $|\gamma_{2}(t)\sigma(t)| \leq c(\delta)$.
\item For $t \in (-\infty,a_{i}]$ we have $b(\gamma_{i}(t)) \doteq_{c(\delta)} t+(x_{1}|x_{2})_{b}$.
\item For $t \in I$ we have $b(\sigma(t)) \doteq_{c(\delta)} |t| + (x_{1}|x_{2})_{b}$. 
\end{enumerate} 
\end{prop}

\begin{proof}
We identify $L_{1} \cup L_{2}$ and $L_{1} \cup L_{3}$ with $\R$ by reversing the orientation of $L_{1}$ and identifying it with $(-\infty,0]$. We then take $(-\infty,a_{i}]$ to be the image under $T$ of $\omega x_{i}$ for $i = 1,2$ and take $\gamma_{i}:(-\infty,a_{i}] \rightarrow X$ to be the inverse of $T$ restricted to the edge $\omega x_{i}$. Similarly for the edge $x_{1}x_{2}$ we identify $L_{2} \cup L_{3}$ with $\R$ by reversing the orientation of $L_{2}$ and identifying it with $(-\infty,0]$ in this case. We set $I$ to be the image of $x_{1}x_{2}$ under $T$ and define $\sigma: I \rightarrow X$ to be the inverse of $T$ restricted to $x_{1}x_{2}$. Properties (1)-(5) of these parametrizations can then be verified directly from the conclusions of Proposition \ref{compute Busemann} and the fact that $T$ is $400\delta$-roughly isometric. 
\end{proof}




We'll need a variant of Proposition \ref{compute Busemann} when we consider $\omega$ as a point of $X$ instead. 

\begin{lem}\label{interior basepoint}
Let $\Delta = xyz$ be a geodesic triangle in $X$ with $x \in X$ and let $T: \Delta \rightarrow \Upsilon$ be a tripod map associated to $\Delta$. Let $\bar{y} \in xy$ and $\bar{z} \in xz$ be the points such that $T(\bar{y}) = T(\bar{z}) = o$. Then
\[
|x\bar{y}| \doteq_{c(\delta)} (y|z)_{x} \doteq_{c(\delta)} |x\bar{z}|,
\]
and consequently,
\begin{equation}\label{rough product}
(y|z)_{x} \doteq_{c(\delta)} \dist(x,yz).
\end{equation}
\end{lem}

\begin{proof}
Let $T: \Delta \rightarrow \Upsilon$ be a tripod map associated to $\Delta$. By \eqref{sequence approximation} we can find $y' \in \bar{y}y$, $z' \in \bar{z}z$ such that $(y'|z')_{x} \doteq_{c(\delta)} (y|z)_{x}$. Then $(T(y')|T(z'))_{T(x)} \doteq_{c(\delta)} (y'|z')_{x}$. Since $T(y') \in L_{2}$ and $T(z') \in L_{3}$, a quick calculation shows that 
\[
(T(y')|T(z'))_{T(x)} = |T(x)o| \doteq_{c(\delta)} |x\bar{y}| \doteq_{c(\delta)} |x\bar{z}|.
\] 
This gives the first assertion. The second follows from the combined observations that $T$ is $c(\delta)$-roughly isometric, that $o \in T(yz) \subset L_{2} \cup L_{3}$, and that $|T(x)o| = \dist(T(x),L_{2} \cup L_{3})$.
\end{proof}

Inequality \eqref{rough product} actually holds with $c(\delta) = 8\delta$, see \cite[(3.2)]{BHK}.  

Lastly we record the following useful inequality regarding geodesic rays. The lemma below rephrases the conclusions of \cite[Lemma 3.6]{Bu20}.

\begin{lem}\label{infinite triangle}
Let $\gamma,\sigma:[0,\infty) \rightarrow X$ be geodesic rays with the same endpoint in $\p X$. Then for all $t \geq 0$ we have
\[
|\gamma(t)\sigma(t)| \leq 3|\gamma(0)\sigma(0)|+8\delta. 
\]
\end{lem}

\subsection{Rough starlikeness} 
In this section we will prove Proposition \ref{all star}. For a proper geodesic $\delta$-hyperbolic space $X$ and any $x \in X$ we define
\[
S(x)=\sup_{\omega \in \p X} \inf_{\xi \in \p X} (\omega|\xi)_{x}.
\]
Despite its appearance, it is actually fairly easy to compute upper bounds on $S(x)$ as computing the Gromov product of any two distinct points in $\p X$ will always produce an upper bound. We will use the explicit constant $c(\delta) = 8\delta$ noted after the statement of \eqref{Gromov inequality}.

\begin{lem}\label{two points}
For any $\omega,\xi \in \p X$ we have
\[
S(x) \leq (\omega|\xi)_{x} + 8\delta.
\]
In particular $S(x) < \infty$ if $\p X$ contains at least two points. 
\end{lem}

\begin{proof}
Define $S(x,\omega) = \inf_{\xi \in \p X} (\omega|\xi)_{x}$ for each $\omega \in \p X$. By fixing $\omega, \zeta \in \p X$ and minimizing \eqref{Gromov inequality} over all $\xi \in \p X$, we obtain
\[
S(x,\omega) \geq \min\{S(x,\zeta),(\omega|\zeta)_{x}\} - 8\delta = S(x,\zeta) - 8\delta,
\]
by the definition of $S(x,\zeta)$. By maximizing this inequality over all $\zeta \in \p X$ we then obtain
\[
S(x) \leq S(x,\omega)+8\delta,
\]
which holds for any $\omega \in \p X$. This gives the claim of the lemma. 
\end{proof}

The next proposition summarizes the relations between rough starlikeness from points of $X$ and rough starlikeness from points of $\p X$.

\begin{prop}\label{rough star boundary}
Let $X$ be a proper geodesic $\delta$-hyperbolic space such that $\p X$ contains at least two points. Let $K \geq 0$ be given. Then
\begin{enumerate}
\item If $X$ is $K$-roughly starlike from $\omega \in \p X$ then it is $(K+10\delta)$-roughly starlike from any point $x \in  X \cup \p X$. 
\item If $X$ is $K$-roughly starlike from $x \in X$ then it is $(K+S(x)+c(\delta))$-roughly starlike from any point $\omega \in \p X$. 
\end{enumerate}
\end{prop}

\begin{proof}
We first assume that $X$ is $K$-roughly starlike from a point $\omega \in \p X$ and let $x \in X \cup \p X$ be given; we may assume that $x \neq \omega$ as otherwise the claim is trivial. Let $p \in X$ and let $\xi \in \p X$ be such that $\dist(p,\xi\omega) \leq K$ for some choice of  geodesic $\xi\omega$ from $\xi$ to $\omega$. We may assume that $\xi \neq x$, as when $x = \xi$ we immediately obtain the desired estimate $\dist(p,x\omega) \leq K$.  Then let $y \in \xi\omega$ satisfy $|py| \leq K$. We form a geodesic triangle $\Delta = x\xi\omega$ that includes the geodesic $\xi \omega$ containing $y$. Since this triangle is $10\delta$-thin by the discussion in Section \ref{subsec:defn}, we can find $z \in x\xi \cup x\omega$ such that $|yz| \leq 10\delta$, from which it follows that $|pz| \leq K+10\delta$. Regardless of whether $z \in x\xi$ or $z \in x\omega$, we obtain the necessary estimate to conclude that $X$ is $K$-roughly starlike from $x$.

Now assume that $X$ is $K$-roughly starlike from a point $x \in X$ and let $\omega \in \p X$ be given. Since we assumed that $\p X$ has at least two points, by Lemma \ref{two points} we then have $S(x) < \infty$. Let $p \in X$ be given. Let $x\xi$ denote a geodesic ray from $x$ to a point $\xi \in \p X$ such that $\dist(p,x\xi) \leq K$. We will assume for now that $\xi \neq \omega$.  We choose a point $\zeta \in \p X$ such that
\[
(\omega|\zeta)_{x} \leq S(x)+\delta, 
\]
and consider additional geodesics $x\zeta$ and $\omega\zeta$. We consider two  geodesic triangles: a triangle $\Delta_{1} = x\xi\omega$ containing the geodesic ray $x\xi$ that has $\omega$ as a vertex and a triangle $\Delta_{2} = x\zeta\omega$ that shares the edge $x\omega$ with $\Delta_{1}$.  We let $T_{i}: \Delta_{i} \rightarrow \Upsilon$ be tripod maps associated to each of these triangles for $i=1,2$. 

We let $u \in x\xi$ be the equiradial point for $T_{1}$ on this edge of $\Delta_{1}$. Let $y$ be a point on $x\xi$ such that $|py| \leq K$. If $y \in u\xi$ then the fact that $T_{1}$ is $c(\delta)$-roughly isometric implies  that we can find $z \in \omega \xi$ such that $|yz| \leq c(\delta)$. Then $|zp| \leq K + c(\delta)$, so $\dist(p,\omega\xi) \leq K+c(\delta)$. This implies our desired estimate.

On the other hand, if $y \in xu$ then we let $w_{1} \in x\omega$ be the equiradial point for $T_{1}$ on this edge and let $z \in xw_{1}$ be such that $|xy| = |xz|$. Since $T_{1}$ is $c(\delta)$-roughly isometric we then have $|yz| \leq c(\delta)$. We now let $w_{2} \in x\omega$ and $s \in \omega\zeta$ be the equiradial points of $T_{2}$ on these edges. Then by Lemma \ref{interior basepoint} we have
\[
|xw_{2}| \doteq_{c(\delta)} (\omega|\zeta)_{x} \leq S(x)+\delta.
\]
Thus if $z \in xw_{2}$ then it follows that
\begin{align*}
|ps| &\leq |py| + |yz| + |zs| \\
&\leq K + c(\delta) + |zs| \\
&\leq K + |zw_{2}| + c(\delta)\\
&\leq K + |xw_{2}| +c(\delta)\\
&\leq K + S(x) + c(\delta).
\end{align*}
This gives the necessary estimate for $p$ since $s \in \omega\zeta$. If $z \in w_{2}\omega$ instead then since $T_{2}$ is $c(\delta)$-roughly isometric we can find some $t \in \omega\zeta$ such that $|zt| \leq c(\delta)$. Then
\[
|pt| \leq |py| + |yz| + |zt| \leq K + c(\delta).
\]
This gives the desired estimate in this case as well. Combining all of these claims together, we conclude that $X$ is $(K + S(x) + c(\delta))$-roughly starlike from $\omega$. 

In the case $\xi = \omega$ we can skip the part of this argument concerning the triangle $\Delta_{1}$ and just apply all of our arguments to the triangle $\Delta_{2}$ formed by $x\omega$, $x\zeta$, and $\omega\zeta$, where the geodesic $x\omega$ is now chosen such that $\dist(p,x\omega) \leq K$. We obtain from these arguments the same conclusion as above that $\dist(p,\omega\zeta) \leq K+S(x)+c(\delta)$. 
\end{proof}

Proposition \ref{all star} follows immediately from Proposition \ref{rough star boundary}, since for $x \in \p X$ it is an immediate consequence of (1) and for $x \in X$ we can apply (2) first to get some $K_{0} \geq 0$ and $\omega \in \p X$ such that $X$ is $K_{0}$-roughly starlike from $\omega$, and then we can use (1) to conclude that $X$ is $K'$-roughly starlike from all points of $X \cup \p X$ with $K' = K_{0}+10\delta$. 

A simple example shows that the dependence on $S(x)$ cannot be removed in (2) of Proposition \ref{rough star boundary}, as we demonstrate below. 

\begin{ex}\label{dependence}
Let $t > 0$ be given and let $X_{t} = \R \cup_{0 \sim 0} [0,t]$ be the geodesic metric space obtained by identifying $0 \in \R$ with $0 \in [0,t]$. Let $p$ denote the origin in $\R$ considered as a point of $X_{t}$ and let $x$ denote the point $t$ in $[0,t]$ considered as a point of $X_{t}$. Then $X_{t}$ is $0$-hyperbolic and $0$-roughly starlike from $x$. It has two points $\zeta_{i}$, $i=1,2$ in its Gromov boundary corresponding to the half-lines $[0,\infty)$ and $(-\infty,0]$ in $\R$, which are the only geodesic rays in $X_{t}$ starting from $p$. Furthermore the only  geodesic line in $X_{t}$ is the isometrically embedded copy of $\R$ to which we glued the interval $[0,t]$. Thus $X_{t}$ is $t$-roughly starlike from $p$ and from either point of $\p X_{t}$, but is not $K$-roughly starlike for any $K < t$ from these same points. Note in this example that $(\zeta_{1}|\zeta_{2})_{x} = t$, hence $S(x) = t$. Observe that if we let $t \rightarrow \infty$ then the limiting space $X_{\infty}$ is isometric to the tripod $\Upsilon$, which is $0$-roughly starlike from all of its points; the creation of a new point in the Gromov boundary prevents the limiting space from contradicting Proposition \ref{rough star boundary}.
\end{ex}


When $\p X$ consists of a single point there are often simpler direct arguments available when $X$ is roughly starlike from one of its points. In this case $X$ can be thought of as the half-line $[0,\infty)$ up to a bounded error, as is shown below. 

\begin{prop}\label{rough ray}
Let $X$ be a proper geodesic $\delta$-hyperbolic space such that $X$ is $K$-roughly starlike from a point $x \in X$ and $\p X = \{\omega\}$ consists of a single point. Then for any geodesic ray $x\omega$ we have $\dist(p,x\omega) \leq K'$ for all $p \in X$ with $K' = K + 8\delta$. 
\end{prop}

\begin{proof}
Let $\gamma,\sigma: [0,\infty) \rightarrow X$ be two geodesic rays starting at $x$. By Lemma \ref{infinite triangle} we have $|\gamma(t)\sigma(t)| \leq 8\delta$ for all $t \geq 0$. The proposition follows from this inequality since any point $p \in X$ is within distance $K$ of some geodesic ray starting from $x$ by the $K$-rough starlikeness assumption from $x$. 
\end{proof}

\subsection{Uniformization estimates} We summarize in this section the estimates we will need from our previous work \cite{Bu20} regarding uniformization of Gromov hyperbolic spaces. We let $X$ be a proper $\delta$-hyperbolic geodesic metric space and let $b \in \hat{\mathcal{B}}(X)$ with basepoint $\omega$. We let $\rho_{\e,b}(x) = e^{-\e b(x)}$ be the associated density on $X$ for a given $\e > 0$. We write $d_{\e,b}$ for the metric on the conformal deformation $X_{\e,b}$ of $X$ with conformal factor $\rho_{\e,b}$.  

Since $b$ is $1$-Lipschitz we have the \emph{Harnack type inequality} for $x,y \in X$,
\begin{equation}\label{Harnack}
e^{-\e|xy|} \leq \frac{\rho_{\e,b}(x)}{\rho_{\e,b}(y)} \leq e^{\e|xy|}.
\end{equation}
Integrating this inequality over a curve joining $x$ to $y$ gives the following inequality for $x,y \in X$ \cite[Lemma 4.2]{Bu20}, 
\begin{equation}\label{arc Harnack}
\rho_{\e,b}(x)\e^{-1}(1-e^{-\e |xy|}) \leq d_{\e,b}(x,y) \leq \rho_{\e,b}(x)\e^{-1}(e^{\e |xy|}-1).
\end{equation}

Lemma \ref{lem:estimate both} below sharpens the estimate \eqref{arc Harnack} when $\rho_{\e,b}$ is a GH-density.  We impose the convention that $|xy| = \infty$ if $x \neq y$ and either $x \in \p X$ or $y \in \p X$, and that $|xy| = 0$ if $x = y \in \p X$. We recall that $\p_{\omega}X = \p X$ if $\omega \in X$. 

\begin{lem}\label{lem:estimate both}\cite[Lemma 4.7]{Bu20}
Suppose that $\rho_{\e,b}$ is a GH-density with constant $M$. Then for each $x,y \in X \cup \p_{\omega}X$ we have
\begin{equation}\label{estimate both}
d_{\e,b}(x,y) \asymp_{C} e^{-\e (x|y)_{b}}\min\{1,|xy|\},
\end{equation} 
with $C = C(\delta,\e,M)$. 
\end{lem}

This lemma is stated under an additional hypothesis on $X$ in \cite{Bu20} that is always satisfied when $X$ is proper, as is pointed out after \cite[Definition 4.6]{Bu20}. The hypothesis that $\rho_{\e,b}$ is a GH-density is only needed for the lower bound in \eqref{estimate both}. The upper bound always holds for a constant $C = C(\delta,\e)$ by \cite[Lemma 4.5]{Bu20}.

We write $d_{\e,b}(x) = d_{X_{\e,b}}(x)$ for $x \in X$. We then have the following fundamental estimate. 

\begin{lem}\label{compute distance}\cite[Lemma 4.15]{Bu20}
Suppose that $\rho_{\e,b}$ is a GH-density constant $M$ and that $X$ is $K$-roughly starlike from $\omega$. Then for $x \in X$ we have
\begin{equation}\label{compute distance inequality}
d_{\e,b}(x) \asymp_{C} \rho_{\e,b}(x),
\end{equation}
with $C = C(\delta,K,\e,M)$.
\end{lem}

Lemmas \ref{lem:estimate both} and \ref{compute distance} are stated for $b \in \mathcal{B}(X)$ in \cite{Bu20}, however as noted in \cite[Remark 4.24]{Bu20} the estimates for $b \in \mathcal{D}(X)$ can be deduced from the estimates for $b \in \mathcal{B}(X)$ by attaching a ray to $X$ at the basepoint of a given $b \in \mathcal{D}(X)$. 

We will need to use the following claim for $b \in \mathcal{B}(X)$ in the proof of Theorem \ref{inversion theorem}. This claim formalizes the intuition that the basepoint $\omega$ of $b$ serves as the ideal point at infinity for the uniformization $X_{\e,b}$. 

\begin{lem}\label{go to infinity}
Suppose that $b \in \mathcal{B}(X)$, let $\omega$ be the basepoint of $b$, and let $z \in X$ be given. Suppose that $X$ is roughly starlike from $\omega$ and that $\rho_{\e,b}$ is a GH-density.  Then for a sequence $\{x_{n}\} \subset X$ we have $d_{\e,b}(z,x_{n}) \rightarrow \infty$ if and only if $(x_{n}|\omega)_{z} \rightarrow \infty$. 
\end{lem}

\begin{proof}
By Proposition \ref{all star} and Theorem \ref{Gehring-Hayman} we can find $\e_{0} > 0$ and $K_{0} \geq 0$ such that $X$ is $K_{0}$-roughly starlike from $z$ and such that $\rho_{\e_{0},z}$ is a GH-density with constant $M = 20$. As remarked in \cite[Remark 4.14(b)]{BHK}, the completion $\bar{X}_{\e_{0},z}$ of $X_{\e_{0},z}$ is compact and can be identified with $X \cup \p X$, with a neighborhood basis of a point $\xi \in \p X_{\e_{0},z}$ in the topology on $X \cup \p X$ induced by this identification being given by
\begin{equation}\label{neighborhood}
N_{\omega,\la} = \{x \in X \cup \p X:(x|\omega)_{z} \geq \la\},
\end{equation}
for $\la \geq 0$. 

We first suppose that $d_{\e,b}(z,x_{n}) \rightarrow \infty$. When we consider the sequence $\{x_{n}\}$ in $\bar{X}_{\e_{0},z}$ it must have at least one limit point since $\bar{X}_{\e_{0},z}$ is compact. We claim that $\omega$ is the only possible limit point of $\{x_{n}\}$ in $\bar{X}_{\e_{0},z}$, from which it follows that $x_{n} \rightarrow \omega$ in $\bar{X}_{\e_{0},z}$. The description \eqref{neighborhood} of the neighborhood basis at $\omega$ then implies that $(x_{n}|\omega)_{z} \rightarrow \infty$.

Let $\xi \in \bar{X}_{\e_{0},z}$ be a limit point of $\{x_{n}\}$ in $\bar{X}_{\e_{0},z}$ and let $\{y_{n}\}$ be a subsequence of $\{x_{n}\}$ that converges to $\xi$ in $\bar{X}_{\e_{0},z}$. By the identification $\bar{X}_{\e_{0},z} \cong X \cup \p X$ we can then consider $\xi$ as a point of $X \cup \p X$. If $\xi \in X$ then the sequence $\{y_{n}\}$ also converges to $\xi$ in $X$ by \eqref{arc Harnack}, hence it also converges to $\xi$ in $X_{\e,b}$ by this same inequality for $b$. This contradicts the assumption that $d_{\e,b}(z,y_{n}) \rightarrow \infty$. We thus must have $\xi \in \p X_{\e_{0},z} \cong \p X$. The description \eqref{neighborhood} of the neighborhood basis shows that we must then have $(y_{n}|\xi)_{z} \rightarrow \infty$ as $n \rightarrow \infty$. If $\xi \neq \omega$ then this implies that $(y_{n}|\xi)_{b} \rightarrow \infty$ by \cite[Proposition 3.4.1]{BS07}. Lemma \ref{lem:estimate both} then shows that $d_{\e,b}(y_{n},\xi) \rightarrow 0$. This once again contradicts the fact that $d_{\e,b}(z,y_{n}) \rightarrow \infty$. Thus the only possibility is $\xi = \omega$. 

Now suppose that  $(x_{n}|\omega)_{z} \rightarrow \infty$. The description \eqref{neighborhood} of the neighborhood basis at $\omega$ then implies that $x_{n} \rightarrow \omega$ in $\bar{X}_{\e_{0},z}$. Since $X_{\e,b}$ is a uniform metric space, its completion $\bar{X}_{\e,b}$ is proper by \cite[Proposition 2.20]{BHK}. Thus if the sequence of distances $\{d_{\e,b}(z,x_{n})\}$ were bounded then we could find a point $\xi \in \bar{X}_{\e,b}$ and a subsequence $\{y_{n}\}$ such that $y_{n} \rightarrow \xi$ in $\bar{X}_{\e,b}$. 

If we had $\xi \in X_{\e,b}$ then it would follow that $y_{n} \rightarrow \xi$ in $X_{\e_{0},z}$ since the metrics on $X_{\e,b}$ and $X_{\e_{0},z}$ are locally biLipschitz. But this contradicts the fact that $y_{n} \rightarrow \omega$ in $\bar{X}_{\e_{0},z}$. Hence the only possibility is $\xi \in \p X_{\e,b}$. The identification $\p X_{\e,b} \cong \p_{\omega}X$ then allows us to consider $\xi$ as a point of the Gromov boundary relative to $\omega$. The comparison \eqref{estimate both} implies that $(y_{n}|\xi)_{b} \rightarrow \infty$ since $d_{\e,b}(y_{n},\xi) \rightarrow 0$. By \cite[Proposition 3.4.1]{BS07} this implies that $(y_{n}|\xi)_{z} \rightarrow \infty$. But this contradicts the fact that  $(y_{n}|\omega)_{z} \rightarrow \infty$ since $\xi \neq \omega$. 
\end{proof}

\section{Quasihyperbolization}\label{sec:quasihyperbolize}

In this section we will study the quasihyperbolizations of unbounded uniform metric spaces. Our focus will be on generalizing the results in \cite[Chapter 3]{BHK} to the unbounded setting. Previously the unbounded setting has been treated using the method of sphericalization \cite{HSX08} to reduce claims to the bounded setting.

We let $(\Omega,d)$ be an $A$-uniform metric space (either bounded or unbounded for now) and write $Y = (\Omega,k)$ for the quasihyperbolization of $\Omega$, defined using the quasihyperbolic metric \eqref{quasihyperbolic metric} on $\Omega$. Then $Y$ is a proper geodesic $\delta$-hyperbolic space by \cite[Theorem 3.6]{BHK} with $\delta = \delta(A)$ depending only on $A$. For clarity we will denote distances between points $x,y \in Y$ by $k(x,y)$ as opposed to the standard distance notation $|xy|$ in $\delta$-hyperbolic spaces. We will refer to geodesics in $Y$ as \emph{quasihyperbolic geodesics} in $\Omega$. We denote the distance to the metric boundary of $\Omega$ by $d(x):=d_{\Omega}(x)$ for $x \in \Omega$. By \cite[Proposition 2.20]{BHK}  the completion $\bar{\Omega}$ of $\Omega$ with respect to the metric $d$ is proper. 

By \cite[(2.4)]{BHK} and \cite[(2.16)]{BHK} the quasihyperbolic metric $k$ satisfies the inequality for $x,y \in \Omega$, 
\begin{equation}\label{hyperbolize comparison}
\log\left(1+ \frac{d(x,y)}{\min\{d(x),d(y)\}}\right) \leq k(x,y) \leq 4A^{2} \log\left(1+ \frac{d(x,y)}{\min\{d(x),d(y)\}}\right).  
\end{equation}
We observe the easily verified inequality for any incomplete metric space $(\Omega,d)$ and $x,y \in \Omega$, 
\begin{equation}\label{lip boundary distance}
|d(x)-d(y)| \leq d(x,y),
\end{equation}
which leads to the following inequality for the quasihyperbolic metric for $x,y \in \Omega$ \cite[(2.3)]{BHK},
\begin{equation}\label{BHK 2.3}
\left|\log \frac{d(x)}{d(y)}\right| \leq k(x,y).
\end{equation}

For a rectifiable curve $\gamma$ we let $\l_{d}(\gamma)$ denote its length measured in the metric $d$ and we let $\l_{k}(\gamma)$ denote its length  measured in the quasihyperbolic metric $k$. We note by \cite[(2.15)]{BHK} that if $\gamma:I \rightarrow \Omega$ is an $A$-uniform curve then 
\begin{equation}\label{BHK 2.15}
\l_{k}(\gamma) \leq 4A\log \left(1+\frac{\l_{d}(\gamma)}{\min\{d(x), d(y)\}}\right),
\end{equation}
By \cite[Theorem 2.10]{BHK}, there is a constant $A_{*} = A_{*}(A) \geq 1$ such that quasihyperbolic geodesics $\gamma: I \rightarrow Y$ defined on compact intervals $I \subset \R$ are $A_{*}$-uniform curves in $\Omega$; an explicit bound $A_{*} \leq e^{1000A^{6}}$ is given in \cite{BHK}. The following straightforward lemma allows us to conclude that all quasihyperbolic geodesics are $A_{*}$-uniform curves in $\Omega$. 

\begin{lem}\label{subcurve uniform}
Let $(\Omega,d)$ be an incomplete metric space and let $A \geq 1$ be given. Let $\gamma:I \rightarrow \Omega$ be a curve such that for each compact subinterval $J \subset I$ the restriction $\gamma|_{J}$ is an $A$-uniform curve. Then $\gamma$ is an $A$-uniform curve. 
\end{lem}

\begin{proof}
Let $J_{n} = [t_{n}^{-},t_{n}^{+}]$, $n \in \N$, be a sequence of compact subintervals $J_{n} \subset I$ with $J_{n} \subset J_{n+1}$ for each $n$ and $I = \bigcup_{n = 1}^{\infty} J_{n}$. We first suppose that $\gamma$ is rectifiable, i.e., that $\l_{d}(\gamma) < \infty$. The $A$-uniformity of $\gamma|_{J_{n}}$ implies for each $n$ that
\begin{equation}\label{J curve one}
\l_{d}(\gamma|_{J_{n}}) \leq Ad(\gamma(t_{n}^{-}),\gamma(t_{n}^{+})).
\end{equation}
Letting $n \rightarrow \infty$ in this inequality gives \eqref{uniform one} for $\gamma$. Similarly for each $t \in I$ we will have $t \in J_{n}$ for sufficiently large $n$, and inequality \eqref{uniform two} then follows from letting $n \rightarrow \infty$ in the corresponding inequality for $\gamma|_{J_{n}}$. 

Now suppose that $\l_{d}(\gamma) = \infty$ and write $t_{-} \in [-\infty,\infty)$, $t_{+} \in (-\infty,\infty]$ for the endpoints of $I$. The inequality \eqref{J curve one} applied to compact subintervals $J(s,t) \subset I$ with given endpoints $s < t \in I$ then shows that $d(\gamma(s),\gamma(t)) \rightarrow \infty$ as  $s \rightarrow t_{-}$ and $t \rightarrow t_{+}$. For a given $t \in I$ and $n$ large enough that $t \in J_{n}$, the $A$-uniformity of $\gamma|_{J_{n}}$ implies that
\begin{equation}\label{J curve two}
\min\{\l_{d}(\gamma|_{(J_{n})_{\leq t}}), \l_{d}(\gamma|_{(J_{n})_{\geq t}})\}\leq Ad(\gamma(t)). 
\end{equation}
Since $\l_{d}(\gamma) = \infty$, we must have either $\l_{d}(\gamma|_{(J_{n})_{\leq t}}) \rightarrow \infty$ or $\l_{d}(\gamma|_{(J_{n})_{\geq t}}) \rightarrow \infty$ as $n \rightarrow \infty$; by reversing the orientation of $\gamma$ if necessary we can assume without loss of generality that $\l_{d}(\gamma|_{(J_{n})_{\leq t}}) \rightarrow \infty$, which implies that $\l_{d}(\gamma|_{I_{\leq t}}) = \infty$. By taking $n$ large enough that $\l_{d}(\gamma|_{(J_{n})_{\leq t}}) > Ad(\gamma(t))$ we can then conclude from inequality \eqref{J curve two} that
\[
\l_{d}(\gamma|_{(J_{n})_{\geq t}}) \leq Ad(\gamma(t)),
\]
for sufficiently large $n$. By letting $n \rightarrow \infty$ we conclude that $\gamma|_{I_{\geq t}}$ is rectifiable and 
\begin{equation}\label{I curve}
\min\{\l_{d}(\gamma|_{I_{\leq t}}), \l_{d}(\gamma|_{I_{\geq t}})\} = \l_{d}(\gamma|_{I_{\geq t}}) \leq Ad(\gamma(t)).
\end{equation}
Thus $\gamma$ is $A$-uniform. 
\end{proof}

Since any subcurve of a quasihyperbolic geodesic is also a quasihyperbolic geodesic, we conclude from Lemma \ref{subcurve uniform} that quasihyperbolic geodesics are $A_{*}$-uniform curves in $\Omega$. We will use this fact frequently throughout this and subsequent sections, always denoting the corresponding constant by $A_{*}$ to distinguish it from the other constants since this constant plays a special role. 

We now discuss the relationship of the metric boundary $\p \Omega$ of $\Omega$ to the Gromov boundary $\p Y$ of $Y$, as elucidated by Bonk-Heinonen-Koskela in \cite[Proposition 3.12]{BHK}. By \cite[Proposition 3.12]{BHK} any quasihyperbolic geodesic $\gamma:I \rightarrow \Omega$ can be continuously extended to the closure $\bar{I}$ of $I$ and then reparametrized by arclength with respect to $d$ in order to obtain a curve $\sigma: [0,\l_{d}(\gamma)] \rightarrow \bar{\Omega}$ that restricts to an $A_{*}$-uniform curve $\sigma:(0,\l_{d}(\gamma)) \rightarrow \Omega$. Here we allow $\l_{d}(\gamma) = \infty$, in which case the domain of $\sigma$ is taken to be $[0,\infty)$.

In particular any quasihyperbolic geodesic ray $\gamma:[0,\infty) \rightarrow \Omega$ can be reparametrized to give an $A_{*}$-uniform curve $\sigma:[0,\l_{d}(\gamma)) \rightarrow \Omega$ with the same starting point as $\gamma$. All quasihyperbolic geodesic rays $\gamma$ with $\l_{d}(\gamma) = \infty$ are at bounded distance from each other by \cite[Proposition 3.12(a)]{BHK}, hence they define a common point $\omega \in \p Y$ if any such ray exists.   When $\l_{d}(\gamma) < \infty$ the curve $\sigma$ has a well-defined endpoint in $\p \Omega$, and two quasihyperbolic geodesic rays have the same endpoint in $\p \Omega$ if and only if they are at bounded distance from each other by \cite[Proposition 3.12(b)]{BHK}. When $\Omega$ is bounded the case $\l_{d}(\gamma) = \infty$ cannot occur (since all uniform curves have finite length with respect to $d$ in this case) and the map $\p Y \rightarrow \p \Omega$ given by sending a quasihyperbolic geodesic ray to its endpoint in $\p \Omega$ defines a bijection between $\p Y$ and $\p \Omega$ \cite[Proposition 3.12(d)]{BHK}. When $\Omega$ is unbounded quasihyperbolic geodesic rays $\gamma$ with $\l_{d}(\gamma) = \infty$ always exist by \cite[Proposition 3.12(c)]{BHK}, and in this case \cite[Proposition 3.12(d)]{BHK} instead provides a bijection between $\p_{\omega}Y = \p Y \backslash \{\omega\}$ and $\p \Omega$. In this case we remark that a quasihyperbolic geodesic $\gamma$ satisfies $\l_{d}(\gamma) = \infty$  if and only if it has $\omega$ as an endpoint. A quasihyperbolic geodesic $\gamma$ having $\omega$ as an endpoint can always be parametrized as to have the form $\gamma:(-\infty,a] \rightarrow Y$, $a \in (-\infty,\infty]$, with $\gamma(t) \rightarrow \omega$ as $t  \rightarrow -\infty$; in this case the $A_{*}$-uniformity inequality \eqref{uniform two} implies that for all $t \in (-\infty,a]$, 
\begin{equation}\label{simplified uniform}
\l_{d}(\gamma|_{[t,a]}) \leq A_{*} d(\gamma(t)),
\end{equation}
since $\l_{d}(\gamma|_{(-\infty,t]}) = \infty$. 

The fact that any quasihyperbolic geodesic joining two points of $\bar{\Omega}$ (under the identification of $\bar{\Omega}$ with $Y \cup \p Y$ in the bounded case and $Y \cup \p_{\omega}Y$ in the unbounded case) must have finite length in $\Omega$ leads to the following lemma. 

\begin{lem}\label{other height}
Let $x_{1},x_{2} \in \bar{\Omega}$ and let $x_{1}x_{2}$ be a quasihyperbolic geodesic between these points. Then there is a point $p \in x_{1}x_{2}$ such that $d(p) = \sup_{x \in x_{1}x_{2}} d(x)$. Furthermore there is a constant $C = C(A) \geq 1$ such that
\begin{equation}\label{second height}
d(x_{1},x_{2}) \asymp_{C}  d(p). 
\end{equation}
\end{lem}

\begin{proof}
We must have $\l_{d}(x_{1}x_{2}) < \infty$ since $x_{1}$ and $x_{2}$ are both points of $\bar{\Omega}$. Set $a =\l_{d}(x_{1}x_{2})$ and let $\sigma: [0,a] \rightarrow \bar{\Omega}$ be a $d$-arclength parametrization of the closure of $x_{1}x_{2}$ in $\bar{\Omega}$. Then by the compactness of $[0,a]$ and the continuity of the function $t \rightarrow d(\sigma(t))$ we can find $s \in [0,a]$ such that $d(\sigma(s)) = \sup_{t \in [0,a]}d(\sigma(t))$; we then set $p = \sigma(s)$. The point $p$ then satisfies $d(p) = \sup_{x \in x_{1}x_{2}} d(x)$.

By reversing the orientation of $\sigma$ if necessary we can assume without loss of generality that $p \in \sigma([0,\frac{a}{2}])$. Since $\sigma$ is an $A_{*}$-uniform curve we then have 
\[
d(p) \leq \l_{d}(\sigma|_{[0,\frac{a}{2}]})  \leq A_{*}d\left(\sigma\left(\frac{a}{2}\right)\right) \leq A_{*}d(p).
\]
Thus 
\[
\l_{d}(\sigma) = 2\l_{d}(\sigma|_{[0,\frac{a}{2}]}) \asymp_{A_{*}} d(p).
\]
Since $\l_{d}(\sigma) \asymp_{A_{*}} d(x_{1},x_{2})$ by the $A_{*}$-uniformity of $\sigma$, the comparison \eqref{second height} follows. 
\end{proof}

The next proposition establishes rough starlikeness of $Y$ from any point of $Y \cup \p Y$.   When $\Omega$ is unbounded we obtain quantitative control of the rough starlikeness constant from any point in terms of $A$, while when $\Omega$ is bounded we only have control in terms of $A$ for a specific point in $Y$. For other points we need to use a bound on the ratio 
\begin{equation}\label{tightness}
\phi(\Omega):= \frac{\diam \, \Omega}{\diam \, \p \Omega}.
\end{equation}

\begin{prop}\label{uniform starlike}
Let $\Omega$ be an $A$-uniform metric space and let $Y = (\Omega,k)$ be its quasihyperbolization. 
\begin{enumerate}
\item If $\Omega$ is unbounded then there is a constant $K = K(A)$ such that $Y$ is $K$-roughly starlike from any point of $Y \cup \p Y$. 
\item If $\Omega$ is bounded there is a constant $K = K(A)$ such that $Y$ is $K$-roughly starlike from any point $z \in \Omega$ such that $d(z) = \sup_{x \in \Omega}d(x)$. If $\p \Omega$ contains at least two points then there is a constant $K' = K'(A,\phi(\Omega))$ such that $Y$ is $K'$-roughly starlike from any point of $Y \cup \p Y$. 
\end{enumerate} 
\end{prop}

\begin{proof} 
We start with the case that $\Omega$ is unbounded. We let $\omega \in \p Y$ denote the equivalence class of all quasihyperbolic geodesic rays $\gamma$ satisfying $\l_{d}(\gamma) = \infty$. We will first show that $Y$ is $K$-roughly starlike from $\omega$. To this end we let $x \in \Omega$ be given. Since the completion $\bar{\Omega}$ of $\Omega$ is proper, we can find $\xi \in \p \Omega$ such that $d(\xi,x) = d(x)$. Let $\xi\omega$ be a quasihyperbolic geodesic from $\xi$ to $\omega$, and let $\gamma: [0,\infty) \rightarrow \Omega$ be a $d$-arclength reparametrization of this geodesic with $\gamma(0) = \xi$. Since $d(\gamma(t)) \rightarrow \infty$ as $t \rightarrow  \infty$ by \cite[Proposition 3.12(a)]{BHK} and $d(\gamma(0)) = 0$, by continuity we can find some $s > 0$ such that $d(\gamma(s)) = d(x)$. We set $y = \gamma(s)$. By \eqref{uniform two} and the fact that $\l_{d}(\gamma|_{[s,\infty)}) = \infty$ , we then have
\begin{align*}
d(x,y) &\leq d(\xi,y) + d(\xi,x) \\
&\leq\l_{d}(\gamma|_{[0,s]}) + d(x) \\
&\leq A_{*}d(y) + d(x) \\
&= (A_{*}+1)d(x), 
\end{align*}
Thus by \eqref{hyperbolize comparison}, using again that $d(y) = d(x)$, 
\[
k(x,y) \leq 4A^{2}\log \left(1+ \frac{d(x,y)}{d(x)}\right) \leq 4A^{2}\log(2+A_{*}).
\]
Thus $Y$ is $K$-roughly starlike from $\omega$ with $K =4A^{2}\log(2+A_{*})$. Since $Y$ is $\delta$-hyperbolic with $\delta = \delta(A)$, the rough starlikeness claim from all other points of $Y \cup \p Y$ then follows from (1) of Proposition \ref{rough star boundary}.

Now assume that $\Omega$ is bounded. In this case the $K$-rough starlikeness with $K = K(A)$ from points $z \in \Omega$ satisfying  $d(z) = \sup_{x \in \Omega}d(x)$ follows from the proof of \cite[Theorem 3.6]{BHK}, specifically the discussion after \cite[(3.13)]{BHK}. If $\p \Omega$ contains at least two points then the existence of a constant $K'$ such that $Y$ is $K'$-roughly starlike from any point of $Y \cup \p Y$ then follows from Proposition \ref{all star}, with quantitative dependence on $A$ and an upper bound on the Gromov product based at a point $z$ with $d(z) = \sup_{x \in \Omega}d(x)$ of any two distinct points in $\p Y \cong \p \Omega$ following from Lemma \ref{two points}, (2) of Proposition \ref{rough star boundary}, and the fact that $Y$ is $\delta$-hyperbolic with $\delta = \delta(A)$. We thus need to show that an upper bound on the ratio $\phi(\Omega)$ leads to an upper bound on the Gromov product based at $z$ of a specific choice of two points in $\p Y$. 

By the properness of the completion $\bar{\Omega}$ of $\Omega$ we can choose points $x,y \in \p \Omega$ such that 
\[
d(x,y) = \diam \,\p \Omega = \phi(\Omega)^{-1}\diam \, \Omega.
\]
Since these points are distinct we can find a quasihyperbolic geodesic $xy$ joining $x$ to $y$ in $\Omega$. By Lemma \ref{other height} we can then find $p \in \gamma$ such that 
\[
d(p) \asymp_{C(A)} d(x,y) = \phi(\Omega)^{-1}\diam \, \Omega.
\] 
Since $d(p) \leq d(z)$ and $d(p,z) \leq \diam \, \Omega$, it follows from \eqref{hyperbolize comparison} that
\begin{equation}\label{phi up}
k(p,z) \leq \log(1+C(A)\phi(\Omega)). 
\end{equation}
By Lemma \ref{interior basepoint} we have $(x|y)_{z} \doteq_{c(A)} \dist(z,xy)$ since $\delta = \delta(A)$ (with distances measured in $Y$). By combining this with \eqref{phi up} we conclude that
\[
(x|y)_{z} \leq k(p,z)+c(A) \leq \log(1+C(A)\phi(\Omega))+c(A),
\]
which gives the desired estimate for $K'$ by Lemma \ref{two points} and Proposition \ref{rough star boundary}.
\end{proof}

The second part of claim (2) is false when $\p \Omega$ consists of a single point, as the example $\Omega = [0,1)$ shows (since the quasihyperbolization is isometric to $[0,\infty)$).  When $\p \Omega$ contains more than one point a family of examples shows in (2) that $K'$ cannot be taken to be quantitative in $A$ in general, nor in $\diam \, \Omega$ or $\diam \, \p \Omega$ individually. 

\begin{ex}\label{dependence uniform}
Let $r,s > 0$ be positive parameters. We consider the incomplete geodesic metric space $\Omega = (-r,r) \cup_{0 \sim 0}[0,s]$ obtained by gluing the interval $[0,s]$ onto $(-r,r)$ at $0$. Clearly $\Omega$ is $1$-uniform for any choice of $r,s$, with $\diam \, \Omega = \max\{s,r\}+r$ and $\diam \, \p \Omega = 2r$. The quasihyperbolization of $\Omega$ can be computed to be isometric to the space $X_{t} = \R \cup_{0 \sim 0} [0,t]$ considered in Example \ref{dependence} with $t  = \log(1+\frac{s}{r})$. As noted in that example, the space $X_{t}$ is at best $t$-roughly starlike from the origin $0 \in \R$. Since $r$ and $s$ can be chosen freely, we conclude that we can make $t$ arbitrarily large even though $\Omega$ is $1$-uniform. We can also individually make either $\diam \, \Omega$ or $\diam \, \p \Omega$ arbitrarily small or large by choosing the parameters appropriately. We have the bound $\frac{s}{r} \leq 2\phi(\Omega)$ from the formulas for $\diam \, \Omega$ and $\diam \, \p \Omega$, which leads to the estimate $t \leq \log(1+2\phi(\Omega))$ in accordance with (2) of Proposition \ref{uniform starlike}.
\end{ex}

For the rest of this section we will assume that $\Omega$ is unbounded. We let $\omega \in \p Y$ denote the distinguished point corresponding to the equivalence class of all quasihyperbolic geodesic rays $\gamma$ satisfying $\l_{d}(\gamma) = \infty$. The next lemma gives us more precise control over the function $t \rightarrow d(\gamma(t))$ for a quasihyperbolic geodesic $\gamma$ starting from $\omega$. 

\begin{lem}\label{height control}
There are constants $C = C(A) \geq 1$ and $0 < u = u(A) \leq 1$  such that if $\gamma: (-\infty,a] \rightarrow Y$, $a \in (-\infty,\infty]$, is a quasihyperbolic geodesic starting from $\omega$ then for all $t \leq s \leq a$, 
\begin{equation}\label{desired}
e^{-(s-t)} \leq \frac{d(\gamma(s))}{d(\gamma(t))} \leq Ce^{-u(s-t)}.
\end{equation}
\end{lem}

\begin{proof}

Let $t \leq s \leq a$ be given. Put $h(x) = -\log d(x)$. The inequality \eqref{desired} is equivalent to the inequality
\begin{equation}\label{transformed desired}
s-t \geq h(\gamma(s))-h(\gamma(t)) \geq u(s-t) - c,
\end{equation}
with $c=c(A) \geq 0$ depending only on $A$. We will prove inequality \eqref{desired} in the form \eqref{transformed desired}. The left side of \eqref{transformed desired} follows immediately from the fact that $h$ is $1$-Lipschitz in the quasihyperbolic metric $k$ on $\Omega$ by \eqref{BHK 2.3}.

Verifying the right side of \eqref{transformed desired} is more involved. We will first prove this claim in the case $a = \infty$. We let $\xi \in \p \Omega$ denote the endpoint of the quasihyperbolic geodesic ray $\gamma|_{[0,\infty)}$ in $\p \Omega$.   We set $t_{0} = t$ and for each $n \in \N$ we choose $t_{n}$ inductively such that $t_{n} > t_{n-1}$ and 
\[
\l_{d}(\gamma|_{[t_{n-1},t_{n}]}) = \frac{1}{2}\l_{d}(\gamma|_{[t_{n-1},\infty)}).
\]
Note that $\l_{d}(\gamma|_{[t_{n-1},t_{n}]}) > 0$ since $t_{n} > t_{n-1}$ and that  $\l_{d}(\gamma|_{[t_{n-1},\infty)}) < \infty$ since $\xi \neq \omega$. The above equality implies that
\[
\l_{d}(\gamma|_{[t_{n},\infty)}) = \frac{1}{2}\l_{d}(\gamma|_{[t_{n-1},\infty)}),
\]
which gives us the equality for $n \in \N$, 
\begin{equation}\label{cutting}
\l_{d}(\gamma|_{[t_{n-1},t_{n}]}) = 2\l_{d}(\gamma|_{[t_{n},t_{n+1}]}).
\end{equation}
By \eqref{simplified uniform} we then have
\[
\l_{d}(\gamma|_{[t_{n},t_{n+1}]}) \leq \l_{d}(\gamma|_{[t_{n},\infty)}) \leq A_{*} d(\gamma(t_{n})),
\]
We thus obtain from \eqref{BHK 2.15} and the equality \eqref{cutting}, for $n \geq 0$, 
\begin{align*}
t_{n+1}-t_{n} = \l_{k}(\gamma|_{[t_{n},t_{n+1}]}) &\leq 4A \log \left(1+ \frac{\l_{d}(\gamma|_{[t_{n},t_{n+1}]})}{\min\{d(\gamma(t_{n})),d(\gamma(t_{n+1}))\}}\right) \\
&\leq  4A \log \left(1+ \frac{A_{*}\l_{d}(\gamma|_{[t_{n},t_{n+1}]})}{\min\{\l_{d}(\gamma|_{[t_{n},t_{n+1}]}) ,\l_{d}(\gamma|_{[t_{n+1},t_{n+2}]})\}}\right) \\
&= c_{1},
\end{align*}
with $c_{1} = c_{1}(A) > 0$ depending only on $A$. 

For a lower bound on $t_{n+1}-t_{n}$ we observe that, since $\l_{d}(\gamma|_{[t_{n},\infty)}) = 2\l_{d}(\gamma|_{[t_{n},t_{n+1}]})$, we have for all $n \geq 0$, 
\[
d(\gamma(t_{n})) \leq d(\gamma(t_{n}),\xi) \leq 2\l_{d}(\gamma|_{[t_{n},t_{n+1}]}). 
\]
Thus \eqref{hyperbolize comparison} implies, using $t_{n+1}-t_{n} = \l_{k}(\gamma|_{[t_{n},t_{n+1}]})$ as above,
\begin{align*}
t_{n+1}-t_{n} &\geq \log \left(1+ \frac{d(\gamma(t_{n}),\gamma(t_{n+1}))}{\min\{d(\gamma(t_{n})),d(\gamma(t_{n+1}))\}}\right) \\
&\geq \log \left(1+ \frac{A_{*}^{-1}\l_{d}(\gamma|_{[t_{n},t_{n+1}]})}{\min\{\l_{d}(\gamma|_{[t_{n},\infty)}),\l_{d}(\gamma|_{[t_{n+1},\infty)})\}}\right) \\
&= \log \left(1+ \frac{A_{*}^{-1}\l_{d}(\gamma|_{[t_{n},t_{n+1}]})}{2\min\{\l_{d}(\gamma|_{[t_{n},t_{n+1}]}),\l_{d}(\gamma|_{[t_{n+1},t_{n+2}]})\}}\right) \\
&= c_{0}, 
\end{align*}
with $c_{0} = c_{0}(A) > 0$ depending only on $A$. We have thus shown that there are positive constants $c_{0}$ and $c_{1}$ depending only on $A$  such that for all $n \geq 0$, 
\begin{equation}\label{initial}
c_{0} \leq t_{n+1}-t_{n} \leq c_{1}.
\end{equation}

On the other hand we have, by $A_{*}$-uniformity of $\gamma$ and the construction of the subdivision $\{t_{n}\}$ of $[t,a]$, for each $n \geq 0$ and $m \geq 1$,
\begin{align*}
d(\gamma(t_{n+m})) &\leq \l_{d}(\gamma|[t_{n+m},\infty)) \\
&= \frac{1}{2^{m}}\l_{d}(\gamma|[t_{n},\infty)) \\
&\leq \frac{A_{*}}{2^{m}}d(\gamma(t_{n})).
\end{align*}
We choose $m = m(A)$ to be the minimal positive integer such that $\frac{A_{*}}{2^{m}} < 1$. Then $c_{2} = - \log \frac{A_{*}}{2^{m}}$ is positive and we have from above that
\begin{equation}\label{step initial}
h(\gamma(t_{n+m})) \geq h(\gamma_{t_{n}}) + c_{2}. 
\end{equation}
Since $m$ depends only on $A$, we deduce from \eqref{initial} the inequality
\begin{equation}\label{further initial}
c_{0} \leq t_{n+m}-t_{n} \leq c_{1},
\end{equation}
with $c_{0}$ and $c_{1}$ being (possibly different) positive constants still depending only on $A$. 

Since $h$ is 1-Lipschitz in the quasihyperbolic metric $k$, using \eqref{step initial} and \eqref{further initial} (and replacing $c_{0}$ with $c_{0}' = \min\{c_{0},c_{2}\}$ if necessary) we conclude that we also have  
\[
c_{0} \leq h(\gamma(t_{n+m})) - h(\gamma(t_{n})) \leq c_{1},
\]
for each $n \geq 0$. There is thus a $C = C(A) \geq 1$ such that
\[
h(\gamma(t_{n+m})) - h(\gamma(t_{n})) \asymp_{C} t_{n+m} - t_{n}.
\]
By starting with $n = 0$, recalling $t_{0} = t$, and summing this inequality, we obtain for each integer $q \geq 0$ that 
\begin{equation}\label{pre-transition}
h(\gamma(t_{qm}))-h(\gamma(t))  \asymp_{C} t_{qm} - t.
\end{equation}
The inequality \eqref{further initial} also shows that $t_{qm} \rightarrow \infty$ as $q \rightarrow \infty$. There will thus be an integer $q \geq 0$ such that $t_{qm} \leq s < t_{(q+1)m}$, recalling that $t_{0} = t \leq s$. Applying \eqref{further initial} with $n = qm$ then allows us to conclude that 
\begin{equation}\label{transition}
0 \leq s-t_{qm}\leq c_{1}.
\end{equation}
Since $h$ is $1$-Lipschitz in the metric $k$, by combining inequality \eqref{transition} with the comparison \eqref{pre-transition} we obtain that
\begin{align*}
h(\gamma(s))-h(\gamma(t)) &= (h(\gamma(s)) - h(\gamma(t_{qm})) + (h(\gamma(t_{qm})-h(\gamma(t))) \\
&\geq t_{qm}-s + C^{-1}(t_{qm}-t) \\
&\geq C^{-1}(s-t) -c,
\end{align*}
with $c = c(A) \geq 0$ depending only on $A$. This gives the desired inequality upon setting $u = C^{-1}$. 

We now consider the case $a < \infty$. We only need to establish the right side of inequality \eqref{desired}, as the left side has already been deduced from the fact that $h$ is $1$-Lipschitz in the quasihyperbolic metric. By Proposition \ref{uniform starlike} we can find a quasihyperbolic geodesic $\sigma: \R \rightarrow Y$ starting from $\omega$ and parametrized such that $k(\gamma(a),\sigma(a)) \leq K$, with $K = K(A)$. Then Lemma \ref{infinite triangle} implies that for all $t \in (-\infty,a]$ we have
\[
|\gamma(t)\sigma(t)| \leq 3K + 8\delta = c(A). 
\]
Thus by \eqref{BHK 2.3} we have that $d(\gamma(t)) \asymp_{C} d(\sigma(t))$ for all $t \in (-\infty,a]$ with $C = C(A)$. The right side of inequality \eqref{desired} for $\gamma$ then follows from the corresponding side of the inequality for $\sigma$. 
\end{proof}

Let $b: Y \rightarrow \R$ be a Busemann function based at the distinguished point $\omega \in \p Y$. Our next lemma shows that maximization of the distance from $\p \Omega$ and minimization of the Busemann function $b$ occurs at the same point $p$ on a quasihyperbolic geodesic between two points $x_{1}$ and $x_{2}$ of $\bar{\Omega}$, up to a bounded error determined only by $A$. The existence of the point $p$ below is given by Lemma \ref{other height}.

\begin{lem}\label{height transition Busemann}
Let $x_{1},x_{2} \in \bar{\Omega}$ and let $x_{1}x_{2}$ be a quasihyperbolic geodesic joining them. Let $\Delta = \omega x_{1}x_{2}$ be a geodesic triangle with the prescribed vertices and let $T: \Delta \rightarrow \Upsilon$ be an associated tripod map. Let $p \in x_{1}x_{2}$ be such that $d(p) = \sup_{x \in x_{1}x_{2}} d(x)$.  Then for each equiradial point $z \in \Delta$ of $T$ we have $k(p,z)\leq c(A)$ and therefore
\begin{equation}\label{first height}
b(p) \doteq_{c} (x_{1}|x_{2})_{b} \doteq_{c} \inf_{x \in x_{1}x_{2}} b(x),
\end{equation}
with $c = c(A)$. 
\end{lem}

\begin{proof}
Let $x_{1},x_{2} \in \bar{\Omega}$ be given. We will consider these as points of $Y \cup \p_{\omega}Y$ using the discussion at the beginning of this section. We let $\gamma_{i}: (-\infty,a_{i}] \rightarrow X$ and $\sigma: I \rightarrow X$ be the parametrizations of $\omega x_{i}$ for $i = 1,2$ and of $x_{1}x_{2}$ that are given by Proposition \ref{compute consequence} applied to a geodesic triangle $\Delta = \omega x_{1}x_{2}$ with associated tripod map $T: \Delta \rightarrow \Upsilon$. Since $\delta = \delta(A)$ the conclusions of Proposition \ref{compute consequence} hold with $c(\delta) = c(A)$; we will implicitly be using variations of this observation throughout the rest of the proof.  

By (3) of Proposition \ref{compute consequence} we have for $t \in I_{\leq 0}$ that $k(\sigma(t),\gamma_{1}(-t)) \leq c(A)$ and for $t \in I_{\geq 0}$ that $k(\sigma(t),\gamma_{2}(t)) \leq c(A)$. By \eqref{BHK 2.3} we then have for $t\in I_{\leq 0}$,
\[
d(\sigma(t)) \asymp_{C(A)}d(\gamma_{1}(-t)),
\]
and for $t \in I_{\geq 0}$, 
\[
d(\sigma(t)) \asymp_{C(A)} d(\gamma_{2}(t)).
\]
By Lemma \ref{height control} applied to $\gamma_{1}$ and $\gamma_{2}$ with $t = 0$, together with the fact that $d(\gamma_{i}(0)) \asymp_{C(A)} d(\sigma(0))$ for $i = 1,2$ since $k(\gamma_{i}(0),\sigma(0)) \leq c(A)$ by (1) of Proposition \ref{compute consequence}, we thus deduce that for $s \in I$, 
\begin{equation}\label{constraint}
C^{-1}e^{-|s|}d(\sigma(0)) \leq d(\sigma(s)) \leq Ce^{-u|s|}d(\sigma(0)),
\end{equation}
with  $C = C(A) \geq 1$ and $0 < u = u(A) \leq 1$. 

Let $p \in x_{1}x_{2}$ be given such that $d(p) = \sup_{x \in x_{1}x_{2}} d(x)$. Let $t \in I$ be such that $\sigma(t)= p$. Then $d(\sigma(0)) \leq d(\sigma(t))$ and therefore inequality \eqref{constraint} implies that
\[
d(\sigma(0)) \leq d(\sigma(t)) \leq Ce^{-u|s|}d(\sigma(0)),
\]
which implies that $|s| \leq c(A)$. Since $|s| = k(p,\sigma(0))$, combining this with another application of (1) of Proposition \ref{compute consequence} gives that $k(p,z) \leq c(A)$ for each equiradial point $z$ for the tripod map $T$. The rough equality \eqref{first height} follows by Proposition \ref{compute Busemann} and the fact that $b$ is $1$-Lipschitz.  
\end{proof}

Combining Lemmas \ref{height control} and \ref{height transition Busemann} leads to the following key estimate on ratios of distances for points in $\bar{\Omega}$.

\begin{lem}\label{stepping stone quasisymmetric}
Let $x_{i} \in \bar{\Omega}$, $i =1,2,3$ be given distinct points. Let $\Delta_{1} = \omega x_{1}x_{2}$ and $\Delta_{2} = \omega x_{1}x_{3}$ be geodesic triangles that share the edge $\omega x_{1}$ and let $T_{i}: \Delta \rightarrow \Upsilon$ be associated tripod maps. Let $\gamma_{1,i}$ be the parametrizations of $\omega x_{1}$ given by applying Proposition \ref{compute consequence} to $\Delta_{i}$, $i = 1,2$, and define $s$ such that $\gamma_{1,1}(s) = \gamma_{1,2}(0)$. Then 
\[
\frac{d(x_{1},x_{2})}{d(x_{1},x_{3})} \leq \beta(e^{s}),
\]
where $\beta:[0,\infty) \rightarrow [0,\infty)$ is a homeomorphism given by $\beta(t) = C\max\{t,t^{u}\}$ with $C = C(A) \geq 1$ depending only on $A$ and $0 < u = u(A) \leq 1$ being the constant of Lemma \ref{height control}.
\end{lem}

\begin{proof}
Let $p_{1} \in x_{1}x_{2}$ and $p_{2} \in x_{1}x_{3}$ be such that $d(p_{1}) = \sup_{x \in x_{1}x_{2}} d(x)$ and $d(p_{2}) = \sup_{x \in x_{1}x_{3}} d(x)$. Then Lemma \ref{height transition Busemann} implies that $k(p_{1},\gamma_{1,1}(0)) \leq c(A)$ and $k(p_{2},\gamma_{1,2}(0)) \leq c(A)$, which implies that $d(p_{1}) \asymp_{C} d(\gamma_{1,1}(0))$ and $d(p_{2}) \asymp_{C} d(\gamma_{1,2}(0))$ with $C = C(A)$ by \eqref{BHK 2.3}. Combining these claims with Lemma \ref{other height} then gives
\begin{equation}\label{ratio comparison}
\frac{d(x_{1},x_{2})}{d(x_{1},x_{3})} \asymp_{C(A)} \frac{d(\gamma_{1,1}(0))}{d(\gamma_{1,2}(0))} = \frac{d(\gamma_{1,1}(0))}{d(\gamma_{1,1}(s))}
\end{equation}

If $s \leq 0$ then we can apply the right side of inequality \eqref{desired} together with \eqref{ratio comparison} to the quasihyperbolic geodesic $\gamma_{1,1}$ to obtain that
\[
\frac{d(x_{1},x_{2})}{d(x_{1},x_{3})} \leq C\frac{d(\gamma_{1,1}(0))}{d(\gamma_{1,1}(s))} \leq C(e^{s})^{u},
\]
with $C = C(A) \geq 1$ and $0 < u = u(A) \leq 1$. If $s \geq 0$ then we apply the left side of inequality \eqref{desired} and invert the results to get that
\[
\frac{d(x_{1},x_{2})}{d(x_{1},x_{3})} \leq C\frac{d(\gamma_{1,1}(0))}{d(\gamma_{1,1}(s))} \leq Ce^{s},
\]
with $C = C(A)$. Setting $\beta(t) = C\max\{t,t^{u}\}$ and combining the cases $s\leq 0$ and $s \geq 0$ together gives the conclusion of the lemma. 
\end{proof}

When $\Omega$ is bounded it is shown in the final assertion of \cite[Theorem 3.6]{BHK} that the metric $d$ on $\p \Omega$ is quasisymmetrically equivalent to any visual metric $\theta$ on the Gromov boundary $\p Y$ of $Y$ that is based at some point $x \in Y$. We will show that the analogous claim holds in the unbounded case when we consider visual metrics based at a Busemann function $b$ on $Y$ that is itself based at the distinguished point $\omega \in \p Y$. We refer back to Section \ref{subsec:visual} for definitions regarding visual metrics. 

\begin{prop}\label{quasisymmetric boundary}
Let $\theta$ be a visual metric with parameter $\e > 0$ on $\p_{\omega}Y$ based at a Busemann function $b$ with basepoint $\omega$. Then the identification $(\p_{\omega} Y,\theta) \rightarrow (\p \Omega,d)$ is $\eta$-quasisymmetric with $\eta(t) = C\max\{t^{\e^{-1}},t^{\e^{-1}u}\}$, where $C = C(A,\e) \geq 1$ and $0 < u = u(A) \leq 1$ is the constant of Lemma \ref{height control}. 
\end{prop}

\begin{proof}
To ease notation in this proof we will write $\asymp$ and $\doteq$ for $\asymp_{C}$ and $\doteq_{c}$ where the implied constant depends only on $A$. We can assume that $\p \Omega$ has at least three distinct points, as otherwise the claim is vacuously true. We must show that for any three distinct points $\xi_{1}$, $\xi_{2}$, $\xi_{3} \in \p_{\omega} Y$ we have
\[
\frac{d(\xi_{1},\xi_{2})}{d(\xi_{1},\xi_{3})}\leq \eta\left(\frac{\theta(\xi_{1},\xi_{2})}{\theta(\xi_{1},\xi_{3})}\right),
\]
with the control function $\eta$ having the desired form. Since $\theta$ is a visual metric with parameter $\e$ based at $b$, we have for $\xi,\zeta \in \p_{\omega}Y$ that
\[
\theta(\xi,\zeta) \asymp e^{-\e (\xi|\zeta)_{b}},
\]
with implied constant independent of $\e$ and $A$. Thus it suffices to find a control function $\eta$ of the desired form such that for any $\xi_{1}$, $\xi_{2}$, $\xi_{3} \in \p_{\omega} Y$ we have
\begin{equation}\label{original control}
\frac{d(\xi_{1},\xi_{2})}{d(\xi_{1},\xi_{3})} \leq \eta(e^{-\e((\xi_{1}|\xi_{2})_{b}-(\xi_{1}|\xi_{3})_{b})}).
\end{equation}
In fact we need only find a control function $\eta_{0}$ such that for any $\xi_{1}$, $\xi_{2}$, $\xi_{3} \in \p_{\omega} Y$ we have
\begin{equation}\label{modified control}
\frac{d(\xi_{1},\xi_{2})}{d(\xi_{1},\xi_{3})} \leq \eta_{0}(e^{-((\xi_{1}|\xi_{2})_{b}-(\xi_{1}|\xi_{3})_{b})}),
\end{equation}
as then we can set $\eta(t) = \eta_{0}(t^{\e^{-1}})$. Thus it suffices to establish the inequality \eqref{modified control} with $\eta_{0}(t) = C\max\{t,t^{u}\}$ where $C$ depends only on $A$ and $u$ is the constant of Lemma \ref{height control}.

As in the setup of Lemma \ref{stepping stone quasisymmetric}, we let $\Delta_{1} = \omega \xi_{1}\xi_{2}$ and $\Delta_{2} = \omega \xi_{1}\xi_{3}$ be geodesic triangles that share the edge $\omega \xi_{1}$ and let $T_{i}: \Delta \rightarrow \Upsilon$ be associated tripod maps. Let $\gamma_{1,i}$ be the parametrizations of $\omega \xi_{1}$ given by applying Proposition \ref{compute consequence} to $\Delta_{i}$, $i = 1,2$, and define $s$ such that $\gamma_{1,1}(s) = \gamma_{1,2}(0)$. Let $\gamma_{2}$ and $\sigma_{1}$ be the parametrizations of $\omega \xi_{2}$, and $\xi_{1}\xi_{2}$ supplied by applying Proposition \ref{compute consequence} to $\Delta_{1}$, and let $\gamma_{3}$ and $\sigma_{2}$ be the parametrizations of $\omega \xi_{3}$, and $\xi_{1}\xi_{3}$ given by applying Proposition \ref{compute consequence} to $\Delta_{2}$.

By (4) of Proposition \ref{compute consequence} we have that $b(\gamma_{1,1}(0)) \doteq (\xi_{1}|\xi_{2})_{b}$ and $b(\gamma_{1,2}(0)) \doteq (\xi_{1}|\xi_{3})_{b}$  (recall that $\delta = \delta(A)$ so that $c(\delta) = c(A)$). Since $\gamma_{1,1}(s) = \gamma_{1,2}(0)$, (4) of Proposition \ref{compute consequence} gives  $b(\gamma_{1,2}(0)) \doteq s + (\xi_{1}|\xi_{2})_{b}$ and therefore 
\[
e^{-((\xi_{1}|\xi_{2})_{b}-(\xi_{1}|\xi_{3})_{b})} \asymp e^{s}.
\]
By Lemma \ref{stepping stone quasisymmetric} we thus conclude that
\[
\frac{d(\xi_{1},\xi_{2})}{d(\xi_{1},\xi_{3})} \leq \beta(e^{-((\xi_{1}|\xi_{2})_{b}-(\xi_{1}|\xi_{3})_{b})}),
\]
with $\beta(t) = C\max\{t,t^{u}\}$ for $t \geq 0$, with $C= C(A) \geq 1$ and $u$ being the constant of Lemma \ref{height control}. This implies that inequality \eqref{modified control} holds with $\eta_{0}(t) = C\beta(t)$ for $t \geq 0$ for an appropriate constant $C = C(A)$, which implies the proposition. 
\end{proof}

Since any visual metric on $\p Y$ based at a Busemann function $b$ is quasim\"obius to any visual metric on $\p Y$ based at a point $x \in Y$ (as discussed at the end of Section \ref{subsec:visual}), Proposition \ref{quasisymmetric boundary} gives a new proof of a result of Herron, Shanmugalingam, and Xie \cite[Theorem 6.2]{HSX08} that does not require sphericalizing $\Omega$. Recently Zhou has shown that the quasisymmetry claim of Proposition \ref{quasisymmetric boundary} holds without the assumption that $\Omega$ is locally compact \cite[Theorem 1.2]{ZZ20}. 

\section{Uniformizing the quasihyperbolic metric}\label{sec:uniformize quasihyperbolic}
 
Let $(\Omega,d)$ be an $A$-uniform metric space. As in the previous section we let $Y = (\Omega,k)$ be the quasihyperbolization of $\Omega$. We write $d(x):=d_{\Omega}(x)$ for $x \in \Omega$ as in the previous section. If $\Omega$ is bounded then we let $\omega \in \Omega$ be such that $d(\omega) = \sup_{x \in \Omega} d(x)$, and if $\Omega$ is unbounded then we let $\omega \in \p Y$ denote the point corresponding to the equivalence class of all quasihyperbolic geodesic rays $\gamma$ in $Y$ with $\l_{d}(\gamma) = \infty$. When $\Omega$ is bounded we define $b(x) = k(\omega,z)$ and when $\Omega$ is unbounded we let $b$ be a Busemann function based at $\omega \in \p Y$. Proposition \ref{uniform starlike} shows that $Y$ is $K$-roughly starlike from $\omega$ with $K = K(A)$. We will assume $\e > 0$ is given such that $\rho_{\e,b}$ is a GH-density on $Y$ with constant $M$. Since $Y$ is $\delta$-hyperbolic with $\delta = \delta(A)$ and $b$ is $1$-Lipschitz, by Theorem \ref{Gehring-Hayman} there is always an $\e_{0} = \e_{0}(A)$ and such that $\rho_{\e,b}$ is a GH-density with constant $M = 20$ for any $0 < \e \leq \e_{0}$. We write $Y_{\e} = Y_{\e,b}$ for the conformal deformation of $Y$ with conformal factor $\rho_{\e} = \rho_{\e,b}$. Similarly we drop $b$ from the notation and write $d_{\e} = d_{\e,b}$ for the metric on $Y_{\e}$, etc. We write $B_{\e}(x,r) = B_{d_{\e}}(x,r)$ for the ball of radius $r$ centered at $x \in Y_{\e}$ in the metric $d_{\e}$. The notation $(x|y)_{b}$ always indicates the Gromov product in $Y$ of $x$ and $y$ based at $b$. 

We then have the following proposition that generalizes \cite[Proposition 4.28]{BHK}.

\begin{prop}\label{uniformize quasi}
Let $\e > 0$ be such that $\rho_{\e}$ is a GH-density for $Y$ with constant $M$. Then the map $\Omega \rightarrow Y_{\e}$ induced by the identity map on $\Omega$ is $\p$-biLipschitz with data $(L,\la)$ and is $\eta$-quasisymmetric with $L$, $\la$, and $\theta$ depending only on $A$, $\e$, and $M$.
\end{prop}

Proposition \ref{uniformize quasi} gives a direct quantitative relation between $\Omega$ and a uniformization of its quasihyperbolization $Y$ that we will make use of in the proofs of the  main theorems in this paper. As remarked above, by Theorem \ref{Gehring-Hayman} we will always be able to apply Proposition \ref{uniformize quasi} for $\e$ sufficiently small with $M = 20$. Thus there is always a uniformization of $Y$ to which this proposition can be applied.  We will focus primarily on the case that $\Omega$ is unbounded. We will deduce the bounded case from this by attaching a ray to $\Omega$.  

\begin{rem}\label{appendix interlude} Throughout the remainder of this paper we will be using \cite[Proposition A.7]{BHK}, which for a geodesic metric space $X$ and a continuous function $\rho:X \rightarrow (0,\infty)$  allows us to compute the lengths $\l_{\rho}(\gamma)$ in the conformal deformation $X_{\rho}$ of curves $\gamma: I \rightarrow X$ parametrized by arclength in $X$ as
\begin{equation}\label{appendix prop}
\l_{\rho}(\gamma) = \int_{I}\rho \circ \gamma \, ds,
\end{equation}
with $ds$ denoting the standard length element in $\R$.
\end{rem}

We will use the following lemma for verifying that a  map is $\p$-biLipschitz. 

\begin{lem}\label{weak to normal}
Let $f: (\Omega,d) \rightarrow (\Omega',d')$ be a homeomorphism of incomplete metric spaces. Suppose that there is $L \geq 1$ and $0 < \la < 1$ such that for any $x \in \Omega$ and $y,z \in B_{d}(x,\la d_{\Omega}(x))$, 
\begin{equation}\label{later controlled comparison}
\frac{d'(f(y),f(z))}{d_{\Omega'}'(f(x))} \asymp_{L} \frac{d(y,z)}{d_{\Omega}(x)},
\end{equation}
and that for all $x \in \Omega$ we have
\begin{equation}\label{controlled inclusion}
B_{d'}(f(x),L^{-1}\la d_{\Omega'}(f(x))) \subseteq f(B_{d}(x,\la d_{\Omega}(x))). 
\end{equation}
Then $f$ is $\p$-biLipschitz with data $(L,L^{-1}\la)$. 
\end{lem}

\begin{proof}
Inequality \eqref{later controlled comparison} clearly implies that $f$ is $\p$-Lipschitz with data $(L,\la)$. Thus we must show that $f^{-1}$ is $\p$-Lipschitz with data $(L,L^{-1}\la)$. Let $w \in \Omega'$ be given. Applying \eqref{controlled inclusion} to $x = f^{-1}(w)$ and then applying $f^{-1}$ to each side gives
\[
f^{-1}(B_{d'}(w,L^{-1}\la d_{\Omega'}(w))) \subseteq B_{d}(x,\la d_{\Omega}(x))
\]
Thus if $y,z \in B_{d'}(w,L^{-1}\la d_{\Omega'}(w))$ then $f^{-1}(y),f^{-1}(z) \in B_{d}(x,\la d_{\Omega}(x))$. We can thus apply \eqref{later controlled comparison} to obtain 
\[
\frac{d'(y,z)}{d_{\Omega'}'(w)} \asymp_{L} \frac{d(f^{-1}(y),f^{-1}(z))}{d_{\Omega}(f^{-1}(w))}.
\]
This implies in particular that $f^{-1}$ is $\p$-Lipschitz with data $(L,L^{-1}\la)$.
\end{proof}

We can now prove Proposition \ref{uniformize quasi}. 

\begin{proof}[Proof of Proposition \ref{uniformize quasi}]
We will first consider the case that $\Omega$ is unbounded, so that $\omega \in \p Y$ and $b$ is a Busemann function based at $\omega$. We recall that $Y$ is $\delta$-hyperbolic with $\delta = \delta(A)$ and $K$-roughly starlike from $\omega$ with $K = K(A)$. We will first verify that the identity map $\Omega \rightarrow Y_{\e}$ is $\p$-biLipschitz. We begin by noting that there is $\la = \la(A) \in (0,1)$ such that $k(y,z) \leq 1$ for $y,z \in B_{d}(x,\la d(x))$. This is because $d(y,z) \leq 2\la d(x)$ and
\[
\min\{d(y),d(z)\} \geq (1-\la)d(x),
\]
and therefore by \eqref{hyperbolize comparison},
\[
k(y,z) \leq C(A) \log \left(1+ \frac{2\la}{1-\la}\right). 
\]
Thus if $\la$ is sufficiently small, dependent only on $A$, we will have $k(y,z) \leq 1$. We apply inequality \eqref{estimate both} to obtain from this that for $y,z \in B_{d}(x,\la d(x))$,
\begin{equation}\label{help quasi}
d_{\e}(y,z) \asymp_{C} \rho_{\e}(x)k(y,z) \asymp_{C} d_{\e}(x)k(y,z),
\end{equation}
with $C = C(A,\e,M)$, with the second comparison following from Proposition  \ref{compute distance}; we are using here that $(y|z)_{b} \doteq_{2} b(x)$ since $k(x,y) \leq 1$ and $k(x,z) \leq 1$ by our choice of $\la$. 

In this next part we will use the following inequality,
\begin{equation}\label{strange inequality}
a \log(1+t) \geq t \log (1+a),
\end{equation}
valid for $0 \leq t \leq a$ (see \cite[(4.31)]{BHK}), as well as the inequality $\log(1+t)\leq t$ for $t \geq 0$. Applying \eqref{hyperbolize comparison} again to $y,z \in B_{d}(x,\la d(x))$, 
\begin{align*}
k(y,z) &\leq C(A) \log \left(1+\frac{d(y,z)}{(1-\la)d(x)}\right) \\
&\leq C(A)\frac{d(y,z)}{(1-\la)d(x)},
\end{align*}
and by \eqref{hyperbolize comparison} together with \eqref{strange inequality},
\begin{align*}
k(y,z) &\geq \log \left(1+\frac{d(y,z)}{(1+\la)d(x)}\right) \\
&\geq \frac{\log(1+a)}{a}\frac{d(y,z)}{(1+\la)d(x)},
\end{align*}
with 
\[
a = a(A) := \frac{2\la}{1+\la} \geq \frac{d(y,z)}{(1+\la)d(x)}.
\]
We conclude that 
\[
k(y,z)\asymp_{C} \frac{d(y,z)}{d(x)},
\]
with $C = C(A)$. Combining this with \eqref{help quasi} gives
\begin{equation}\label{part controlled}
\frac{d_{\e}(y,z)}{d_{\e}(x)} \asymp_{L} \frac{d(y,z)}{d(x)},
\end{equation}
with $L = L(A,\e,M)$. Putting $z = x$ in \eqref{part controlled} shows that if $y \in B_{d_{\e}}(x,L^{-1}\la d_{\e}(x))$ then $y \in B_{\Omega}(x,\la d(x))$. It then follows from Lemma \ref{weak to normal} that the identity map $\Omega \rightarrow Y_{\e}$ is $\p$-biLipschitz with data $(L,L^{-1}\la)$ depending only on $A$, $\e$, and $M$.  

It remains to show that the identity map $\Omega \rightarrow Y_{\e}$ is $\eta$-quasisymmetric with $\eta$ depending only on $A$, $\e$, and $M$. Since the metric spaces $\Omega$ and $Y_{\e}$ are both uniform, by \cite[Theorem 6.6]{V99} it suffices to show that there is some $C = C(A,\e,M)$ such that, for $x_{1},x_{2},x_{3}\in \Omega$, 
\[
d(x_{1},x_{2}) \leq d(x_{1},x_{3}) \; \Rightarrow \; d_{\e}(x_{1},x_{2}) \leq Cd_{\e}(x_{1},x_{3}).
\]  


Recalling that $\omega \in \p Y$ is the basepoint of the Busemann function $b$, similarly to the proof of Proposition \ref{quasisymmetric boundary}  we form geodesic triangles $\Delta_{1} = x_{1}x_{2}\omega$ and $\Delta_{2}=x_{1}x_{3}\omega$ sharing $x_{1}\omega$ as a common edge and let $T_{i}: \Delta_{i} \rightarrow \Upsilon$ be associated tripod maps. We let $\gamma_{1,i}$ be the parametrizations of $\omega x_{1}$ given by applying Proposition \ref{compute consequence} to $\Delta_{i}$, $i =1,2$, and define $s$ such that $\gamma_{1,1}(s) = \gamma_{1,2}(0)$. We conclude by Lemma \ref{stepping stone quasisymmetric} that
\[
\frac{d(x_{1},x_{2})}{d(x_{1},x_{3})} \leq \beta(e^{s}), 
\]
with $\beta(t) = C(A)\max\{t,t^{u}\}$ for $t \geq 0$, $0 < u = u(A) \leq 1$. Thus it suffices to find a constant $c = c(A,\e,M) \geq 0$ such that we always have $s \leq c$. Since this inequality is trivial when $s \leq 0$ we can assume that $s \geq 0$. 

Let $p_{1} \in x_{1}x_{2}$ and $p_{2} \in x_{1}x_{3}$ be such that $d(p_{1}) = \sup_{x \in x_{1}x_{2}}d(x_{1},x_{2})$ and $d(p_{2}) = \sup_{x \in x_{1}x_{3}}d(x_{1},x_{3})$. Lemma \ref{other height} implies that we have $d(p_{1}) \asymp_{C} d(x_{1},x_{2})$ and $d(p_{2}) \asymp_{C} d(x_{1},x_{3})$ with $C = C(A)$. The inequality $d(x_{1},x_{2}) \leq d(x_{1},x_{3})$ then implies that 
\[
d(p_{1}) \leq C(A)d(p_{2}).
\] 
On the other hand, Lemma \ref{height transition Busemann} implies that $k(p_{i},\gamma_{1,i}(0)) \leq c(A)$ for $i = 1,2$, which implies that $d(p_{i}) \asymp_{C} d(\gamma_{1,i}(0))$ for $i = 1,2$ with $C = C(A)$ by \eqref{BHK 2.3}. Thus we conclude that
\[
 \frac{d(\gamma_{1,2}(0))}{d(\gamma_{1,1}(0))} \geq C^{-1},
\]
with $C = C(A)$ depending only on $A$. Since $s \geq 0$ we can combine this with the right side inequality of Lemma \ref{height control} to conclude that, for $C = C(A) \geq 1$ and $u = u(A) > 0$,  
\[
C^{-1} \leq \frac{d(\gamma_{1,2}(0))}{d(\gamma_{1,1}(0))} \leq Ce^{-us}
\]
Rearranging this inequality gives $s \leq c$ with $c = c(A)$ depending only on $A$, as desired. This completes the proof of Proposition \ref{uniformize quasi} in the case that $\Omega$ is unbounded. 

For the case that $\Omega$ is bounded the corresponding uniformization $Y_{\e} = Y_{\e,b}$  is given by $b \in \mathcal{D}(Y)$ of the form $b(x) = k(x,\omega)$ for a point $\omega \in \Omega$ such that $d(\omega) = \sup_{x \in \Omega} d(x)$. We let $\t{\Omega} = \Omega \cup_{z \sim 0}[0,\infty)$ be the \emph{ray augmentation} of $\Omega$ based at $z \in \Omega$, given by attaching the half-line $[0,\infty)$ to $\Omega$ at $z$ (see \cite[Definition 4.18]{Bu20}). Then $\t{\Omega}$ is clearly also an $A$-uniform metric space with $\p \t{\Omega} = \p \Omega$. The quasihyperbolization $\t{Y}$ of $\t{\Omega}$ is isometric to the ray augmentation $Y \cup_{z \sim 0} [0,\infty)$ of the quasihyperbolization $Y$ of $\Omega$ based at $z$ by \cite[Lemma 4.20]{Bu20}. We let $\t{b}$ be the Busemann function associated to the ray $[0,\infty)$ we glued onto $Y$ and let $\t{\omega} \in \p \t{Y}$ be the basepoint of $\t{b}$. We let $\t{Y}_{\e}$ be the conformal deformation of $\t{Y}$ with conformal factor $\rho_{\e,\t{b}}$. 

As the discussion in \cite[Section 4.3]{Bu20} shows, we can then conclude that $\t{Y}$ is $\delta$-hyperbolic, that $\t{Y}$ is $K$-roughly starlike from $\t{\omega}$, and that  $\rho_{\e,\t{b}}$ is a GH-density for $\t{Y}$ with constant $M$. We can then apply the case $b \in \mathcal{B}(Y)$ that we established above to conclude that the map $\t{\Omega} \rightarrow \t{Y}_{\e}$ induced by the identity map on $\t{\Omega}$ is $\p$-biLipschitz with data $(L,\la)$ and is $\eta$-quasisymmetric, where $L$, $\la$, and $\eta$ depend only on $A$, $\e$ and $M$. Since the natural embeddings $\Omega \rightarrow \t{\Omega}$ and $Y_{\e} \rightarrow \t{Y}_{\e}$ are isometric by \cite[Lemma 4.20]{Bu20} and we have both $\p \Omega = \p \t{\Omega}$ and $\p Y_{\e} = \p \t{Y}_{\e}$, we conclude that the identity map $\Omega \rightarrow Y_{\e}$ is $\p$-biLipschitz with data $(L,\la)$ and is $\eta$-quasisymmetric, where $L$, $\la$, and $\eta$ depend only on $A$, $\e$ and $M$.  
\end{proof}

We are also able to obtain a variant of Proposition \ref{uniformize quasi} if we choose an arbitrary $b \in \hat{\mathcal{B}}(Y)$ as opposed to the specific choice we made in that proposition. The proposition below follows by combining together Proposition \ref{uniformize quasi}, Proposition \ref{uniform starlike} of the previous section, and Propositions \ref{quasimobius uniform} and \ref{composition true} from Section \ref{sec:main theorems}. 

\begin{prop}\label{extend uniformize quasi}
Let $\Omega$ be an $A$-uniform metric space such that $\p \Omega$ contains at least two points if $\Omega$ is bounded. Let $b \in \hat{\mathcal{B}}(Y)$ and $\e > 0$ be given such that $\rho_{\e,b}$ is a GH-density for $Y$ with constant $M$. Then the map $\Omega \rightarrow Y_{\e,b}$ induced by the identity map on $\Omega$ is $\p$-biLipschitz with data $(L,\la)$ and is $\eta$-quasim\"obius with $L$, $\la$, and $\eta$ depending only on $A$, $\e$, and $M$ (and additionally on $\phi(\Omega)$ if $\Omega$ is bounded). 
\end{prop}

\section{Quasihyperbolizing the uniformization}\label{sec:quasi to uniform}

One may also consider the reverse direction from Section \ref{sec:uniformize quasihyperbolic}. We can start with a proper geodesic $\delta$-hyperbolic space $X$ and a function $b \in \hat{\mathcal{B}}(X)$ such that $X$ is $K$-roughly starlike from the basepoint $\omega_{b}$ of $b$, let $\e > 0$ be given such that $\rho_{\e,b}$ is a  GH-density on $X$ with constant $M$, consider the resulting uniformization $X_{\e,b}$ of $X$, and then take the quasihyperbolization $Y$ of $X_{\e,b}$. For the case $b \in \mathcal{D}(X)$ this situation is considered under slightly more restrictive hypotheses in \cite[Proposition 4.37]{BHK}, in which case it is shown that the map $X \rightarrow Y$ induced by the identity map on $X$ is biLipschitz. We will generalize their result to the case of Busemann functions $b \in \mathcal{B}(X)$ here, with nearly the same proof. This result will not be needed elsewhere in the paper. Proposition \ref{biLip circle} is of significantly less theoretical importance than Proposition \ref{uniformize quasi} due to the fact that it is generally preferable to consider the original metric from $X$ as opposed to the quasihyperbolic metric on $X_{\e,b}$. 

\begin{prop}\label{biLip circle}
The identity map $X \rightarrow Y$ is $H$-biLipschitz with $H = H(\delta,K,\e,M)$. 
\end{prop}

\begin{proof}
We will prove this result in the case $b \in \mathcal{B}(X)$. The case $b \in \mathcal{D}(X)$ can then be deduced from this using a ray augmentation argument as we did at the end of the proof of Proposition \ref{uniformize quasi}. Since this result will not be playing an important role in this paper, we leave the details of this deduction to the reader. 

Thus we will assume that $b \in \mathcal{B}(X)$. We write $X_{\e} = X_{\e,b}$, $\rho_{\e} = \rho_{\e,b}$, $d_{\e} = d_{\e,b}$, etc. We write $\l_{\e}(\gamma) = \l_{d_{\e}}(\gamma)$ for the length of a rectifiable curve $\gamma$ in the metric $d_{\e}$.  For $x,y \in X$ we will denote their distance in $X$ by $|xy|$ and their distance in $Y$ by $k(x,y)$. We start by showing that the identity map $X \rightarrow Y$ is Lipschitz. Let $x,y \in X$ and let $\gamma$ be a geodesic in $X$ joining them, parametrized by arc length. Let 
\[
L_{\e}(t) = \int_{0}^{t}\rho_{\e}(\gamma(s))\,ds,
\]
be the length measurement of this curve in $X_{\e}$. Then by Proposition \ref{compute distance}, letting $|d_{\e}z|$ stand for the element of arclength in $(X_{\e},d_{\e})$, 
\begin{align*} 
k(x,y) &\leq \int_{\gamma} \frac{|d_{\e}z|}{d_{\e}(z)} \\
&= \int_{\gamma} \frac{dL_{\e}(t)}{d_{\e}(z)} \\
&= \int_{0}^{|xy|}\frac{\rho_{\e}(\gamma(t))}{d_{\e}(\gamma(t))}\,dt \\
&\leq C |xy|,
\end{align*}
with $C = C(\delta,\e,K,M)$. It follows that the identity map $X \rightarrow Y$ is $C$-Lipschitz. 

For the lower bound on $k(x,y)$, we apply the inequality \eqref{first GH} and the inequality \eqref{hyperbolize comparison} to the geodesic $\gamma$ and then use Lemma \ref{compute distance} as well to obtain
\begin{align*}
k(x,y) &\geq \log \left(1+ \frac{d_{\e}(x,y)}{\min\{d_{\e}(x),d_{\e}(y)\}}\right) \\
&\geq \log \left(1+ C^{-1}\frac{\l_{\e}(\gamma)}{\min\{d_{\e}(x),d_{\e}(y)\}}\right) \\
&\geq \log \left(1+ C^{-1}\frac{\l_{\e}(\gamma)}{\min\{\rho_{\e}(x),\rho_{\e}(y)\}}\right),
\end{align*}
with $C = C(\delta,K,\e,M) \geq 1$. We now restrict to the case $|xy| \leq 1$. By the Harnack inequality \eqref{Harnack} we then have that  $\rho_{\e}(x) \asymp_{e^{\e}} \rho_{\e}(z)$ for all $z \in \gamma$. This implies that 
\[
\l_{\e}(\gamma) \geq e^{-\e}\rho_{\e}(x)|xy|.
\]
Therefore, 
\[
k(x,y) \geq \log \left(1+ C^{-1} |xy|\right),
\]
still with $C = C(\delta,K,\e,M) \geq 1$. Using the inequality \eqref{strange inequality} with $t = C^{-1}|xy|$ and $a = 1$, noting that since $|xy| \leq 1$ and $C \geq 1$ we have $t \leq 1$, we conclude that
\[
k(x,y) \geq C^{-1}|xy|, 
\]
with $C =  C(\delta,K,\e,M)$. This handles the case $|xy| \leq 1$. 

We can thus move to the case $|xy| \geq 1$. We consider a geodesic triangle $\Delta = \omega xy$ with vertices $\omega,x,y$ that has $\gamma$ as an edge. We let $T: \Delta \rightarrow \Upsilon$ be an associated tripod map and assume that $\gamma$ is parametrized in accordance with Proposition \ref{compute consequence}. We set $u = \gamma(0)$. A straightforward computation using (5) of Proposition \ref{compute consequence} shows that
\begin{align*}
\l_{\e}(\gamma) &\geq C^{-1}\rho_{\e}(u)\left(\int_{0}^{|xu|} e^{-\e t}\,dt + \int_{0}^{|uy|} e^{-\e t}\,dt\right) \\
&= C^{-1}\e^{-1}\rho_{\e}(u)(2- e^{-\e |xu|} - e^{-\e |yu|}) \\
&\geq C^{-1}\e^{-1}\rho_{\e}(u)(1- e^{-\e |xy|}) \\
&\geq C^{-1}\rho_{\e}(u),
\end{align*}
with $C = C(\delta,K,\e,M) \geq 1$, using that $|xy| \geq 1$ so that $1-e^{-\e |xy|} \geq 1-e^{-\e}$. 

By (5) of Proposition \ref{compute consequence} we have $|xu| \doteq_{c} b(x)-b(u)$ and $|yu| \doteq_{c} b(y)-b(u)$ with $c = c(\delta)$. By switching the roles of $x$ and $y$ if necessary we can assume that $|xu| \geq |yu|$ and therefore $|xu| \geq \frac{1}{2}|xy|$. Then 
\[
\rho_{\e}(x) \asymp_{C(\delta)} \rho_{\e}(u)e^{-\e |xu|} \leq \rho_{\e}(u)e^{-\frac{\e}{2}|xy|}, 
\] 
so that 
\[
\min\{\rho_{\e}(x),\rho_{\e}(y)\} \leq \rho_{\e}(u)e^{-\frac{\e}{2}|xy|}.
\]
Picking up from the inequality 
\[
k(x,y) \geq \log \left(1+ C^{-1}\frac{\l_{\e}(\gamma)}{\min\{\rho_{\e}(x),\rho_{\e}(y)\}}\right),
\]
it now follows that
\[
k(x,y) \geq \log(1+ C^{-1}e^{\frac{\e}{2}|xy|}),
\]
with  $C = C(\delta,K,\e,M) \geq 1$. We apply the elementary inequality \cite[(2.12)]{BHK}, which states for $a \geq 1$, $t \geq 0$, 
\begin{equation}\label{elementary}
a^{-1}\log(1+at) \leq \log(1+t),
\end{equation}
with $a = C$, $t = C^{-1}e^{\frac{\e}{2}|xy|}$, to obtain 
\[
k(x,y) \geq C^{-1}\log(1+ e^{\frac{\e}{2}|xy|}) \geq C^{-1}|xy|,
\]
with $C = C(\delta,K,\e,M)$. This gives the desired lower bound on $k(x,y)$. 
\end{proof}

\section{The main theorems}\label{sec:main theorems}

In this final section we will complete the proofs of the main theorems. The key step will be Proposition \ref{quasimobius uniform},  which may be of independent interest.

Following Buyalo and Schroeder \cite[Chapter 4]{BS07}, we define the \emph{cross-difference} of four points $x,y,z,w$ in a metric space $X$ by
\begin{equation}\label{cross difference}
\langle x,y,z,w\rangle = \frac{1}{2}(|xz|+|yw|-|xy|-|zw|).
\end{equation} 
The significance of the expression \eqref{cross difference} lies in the following lemma, which is a simple calculation that we leave to the reader.

\begin{lem}\label{cross difference lemma}
Let $X$ be a proper geodesic $\delta$-hyperbolic space. For $b \in \hat{\mathcal{B}}(X)$ we have for any $x,y,z,w \in X$, 
\begin{equation}\label{cross equality}
-(x|z)_{b}-(y|w)_{b} + (x|y)_{b}+(z|w)_{b} = \langle x,y,z,w\rangle.
\end{equation}
\end{lem}

We will also require a lemma regarding the behavior of the distance function $d_{\e,b}$ on balls of the form $B_{d_{\e,b}}(x,\la d_{\e,b}(x))$ for $\la > 0$ sufficiently small.

\begin{lem}\label{subwhitney}
Let $X$ be a proper geodesic $\delta$-hyperbolic space and let $b \in  \hat{\mathcal{B}}(X)$. We suppose that $X$ is $K$-roughly starlike from the basepoint $\omega_{b}$ of $b$ and that we are given $\e > 0$ such that $\rho_{\e,b}$ is a GH-density with constant $M$. Then there exists $0 < \la < 1$ with $\la = \la(\delta,K,\e,M)$ such that for any $x \in X$ and any $y,z \in B_{d_{\e,b}}(x,\la d_{\e,b}(x))$ we have that $|yz| \leq 1$ and that
\begin{equation}\label{subwhitney inequality}
d_{\e,b}(y,z) \asymp_{C} d_{\e,b}(x)|yz|,
\end{equation}
with $C = C(\delta,K,\e,M)$. 
\end{lem}

\begin{proof}
To simplify notation we will write $B_{\e}(x,r) = B_{d_{\e,b}}(x,r)$ for the ball of radius $r$ centered at $x$ in $X_{\e,b}$. We will write $X_{\e} = X_{\e,b}$, $d_{\e} = d_{\e,b}$, etc. Let $0 < \la < 1$ be given. We first observe by inequality \eqref{lip boundary distance} that for $y \in B(x,\la d_{\e}(x))$ we have
\[
d_{\e}(y) \leq d_{\e}(x) + d_{\e}(x,y) \leq (1+\la)d_{\e}(x),
\]
and 
\[
d_{\e}(y) \geq d_{\e}(x) - d_{\e}(x,y) \geq (1-\la)d_{\e}(x). 
\]
By making the restriction $\la \leq \frac{1}{2}$ we can then conclude that $d_{\e}(y) \asymp_{2} d_{\e}(x)$ for $y \in B_{\e}(x,\la d_{\e}(x))$. We will impose this restriction in what follows. 

Thus if $y,z \in B_{\e}(x,\la d_{\e}(x))$ then combining inequality \ref{arc Harnack} with Lemma \ref{compute distance} implies that
\[
1-e^{-\e |yz|} \leq C\frac{d_{\e}(y,z)}{d_{\e}(y)} \leq C\frac{d_{\e}(y,z)}{d_{\e}(x)} \leq C\la,
\]
with $C = C(\delta,K,\e,M)$, which can be rearranged to 
\begin{equation}\label{key subwhitney}
e^{-\e|yz|} \geq 1-C\la.
\end{equation}
Thus by further decreasing $\la$, depending only on $\delta$, $K$, $\e$, and $M$, we can assume that $y,z \in B_{\e}(x,\la d_{\e}(x))$ implies that $e^{-\e |yz|} \geq e^{-\e}$, which implies that $|yz| \leq 1$. This gives the first claim. 

For the second claim we let $\la$ be as determined in the first claim and let  $y,z \in B_{\e}(x,\la d_{\e}(x))$. Since $|yz| \leq 1$, by the comparison \eqref{estimate both} we have
\[
d_{\e}(y,z) \asymp_{C} e^{-\e (y|z)_{b}}|yz|,
\]
with $C = C(\delta,K,\e,M)$. Since $b$ is $1$-Lipschitz we have 
\[
|b(y)-(y|z)_{b}| \leq |yz| \leq 1,
\]
and therefore $\rho_{\e}(y) \asymp_{e^{\e}} e^{-\e (y|z)_{b}}$. Thus
\[
d_{\e}(y,z) \asymp_{C} \rho_{\e}(y)|yz| \asymp_{C} d_{\e}(y)|yz| \asymp_{2}d_{\e}(x)|yz|,
\]
with $C = C(\delta,K,\e,M)$, where we have used Lemma \ref{compute distance} once more. The comparison \eqref{subwhitney inequality} follows. 
\end{proof}

The following proposition is inspired by \cite[Proposition 4.15]{BHK}. However our methods are somewhat different.

\begin{prop}\label{quasimobius uniform}
Let $f:X \rightarrow X'$ be an $H$-biLipschitz map between proper geodesic $\delta$-hyperbolic spaces and let $b \in \hat{\mathcal{B}}(X)$, $b' \in  \hat{\mathcal{B}}(X')$ be given such that $X$ and $X'$ are $K$-roughly starlike from their basepoints $\omega_{b}$ and $\omega_{b'}$ respectively.  We suppose that we are given $\e,\e' > 0$ such that $\rho_{\e,b}$ and $\rho_{\e',b'}$ are GH-densities for $X$ and $X'$ respectively with the same constant $M$. 

Then the induced map $f:X_{\e,b} \rightarrow X_{\e',b'}$ is $\p$-biLipschitz with data $(L,\la)$ and is $\eta$-quasim\"obius with $\eta(t)= C\max\{t^{\beta},t^{\beta^{-1}}\}$. The constants  $L$, $\la$, and $C$ depend only on $\delta$, $K$, $\e$, $\e'$, $M$, and $H$. The exponent $\beta \geq 1$ has the form $\beta = \max\{\frac{\e'}{\e},\frac{\e}{\e'}\}C_{0}$ with $C_{0} = C_{0}(\delta,H)$.
\end{prop}

\begin{proof}

To simplify notation in the proof we will write $x' = f(x)$ for the image of a point under $f:X \rightarrow X'$. We write $d_{\e',b'} = d'_{\e',b'}$, etc. for the uniformized distance on $X'$. We write $B_{\e} = B_{d_{\e,b}}$ for balls in $X_{\e,b}$ and $B_{\e'}' = B_{d'_{\e',b'}}$ for balls in $X_{\e',b'}'$. To avoid repeatedly writing out long lists of parameter dependencies, we will use the expression ``the given data" to indicate the parameters $\delta$, $K$, $\e$, $\e'$, $M$, and $H$.

We will first establish that the induced map $f:X_{\e,b} \rightarrow X_{\e',b'}$ is $\p$-biLipschitz with data $(L,\la)$ depending only on the given data. Since $f^{-1}:X' \rightarrow X$ is also $H$-biLipschitz, it suffices by symmetry to show that $f$ is $\p$-Lipschitz with data $(L,\la)$ depending only on the given data. For use later we will prove this claim under the weaker hypothesis that $f:X \rightarrow X'$ is a map that is $H$-Lipschitz for some $H \geq 0$.

We let $\la'$ be sufficiently small that the conclusions of Lemma \ref{subwhitney} hold on both $B_{\e}(x,\la'd_{\e,b}(x))$ and $B_{\e'}'(x',\la'd_{\e',b'}(x'))$ for $x \in X$; by that lemma we can choose $\la'$ to depend only on the given data. We claim that there is a $0 < \la \leq \la' $ with $\la$ also depending only on the given data such that if $y \in B_{\e}(x,\la d_{\e,b}(x))$ then $y' \in B_{\e'}'(x',\la' d_{\e',b'}(x'))$. The computation of Lemma \ref{subwhitney} (specifically inequality \eqref{key subwhitney}) shows that for any $\kappa > 0$ we can, by sufficiently decreasing $\la$, find $\la$ depending only on $\kappa$ and the given data such that if $y \in B_{\e}(x,\la d_{\e,b}(x))$ then $ |xy| \leq \kappa$. Plugging this into inequality \ref{arc Harnack} applied on $X_{\e',b'}'$ and using Lemma \ref{compute distance} on $X_{\e',b'}'$ together with the fact that $f$ is $H$-Lipschitz, we conclude that
\[
d_{\e',b'}(x',y') \leq C(e^{H\e'\kappa}-1)d_{\e',b'}(x'),
\]
with $C = C(\delta,K,\e',M)$. We then choose $\kappa$ close enough to $0$ that $ C(e^{H\e'\kappa}-1) < \la'$ in the above inequality and then choose $\la \leq \la '$ based on this value of $\kappa$. 

Now let $y,z \in B_{\e}(x,\la d_{\e,b}(x))$ be given. Then by Lemma \ref{subwhitney} we have 
\[
d_{\e,b}(y,z) \asymp_{C} d_{\e,b}(x)|yz|,
\]
with $C = C(\delta,K,\e,M)$. Meanwhile, the work of the previous paragraph shows that $y',z' \in B_{\e'}'(x',\la'd_{\e',b'}(x'))$ as well, which implies by a second application of Lemma \ref{subwhitney} that 
\[
d_{\e',b'}(y',z') \asymp_{C} d_{\e',b'}(x')|y'z'|,
\]
with $C = C(\delta,K,\e',M)$. Since $|y'z'| \leq H|yz|$, the desired inequality \eqref{controlled inequality} immediately follows with a constant $L$ depending only on the given data. This implies that $f$ is $\p$-Lipschitz with data $(L,\la)$ as desired.

We now return to assuming that $f: X \rightarrow X'$ is an $H$-biLipschitz homeomorphism. We will show that $f$ is $\eta$-quasim\"obius with the control function $\eta$ having the indicated form. By the comparison \eqref{estimate both} we have for $x,y \in X$, 
\[
d_{\e,b}(x,y) \asymp_{C} e^{-\e (x|y)_{b}}\min\{1,|xy|\},
\]
with $C = C(\delta,K,\e,M)$. Applying this to the cross-ratio of four distinct points $x,y,z,w \in X_{\e,b}$ and using Lemma \ref{cross difference lemma}, we obtain
\begin{equation}\label{first comparison cross}
[x,y,z,w]_{d_{\e,b}} \asymp_{C} e^{\e \langle x,y,z,w\rangle}\frac{\min\{1,|xz|\}\min\{1,|yw|\}}{\min\{1,|xy|\}\min\{1,|zw|\}},
\end{equation}
with $C = C(\delta,K,\e,M)$. Applying this same calculation to $X'_{\e',b'}$,  recalling that $x' = f(x)$ denotes images of points under $f$, we obtain that
\begin{equation}\label{second comparison cross}
[x',y',z',w']_{d_{\e',b'}} \asymp_{C} e^{\e'\langle x',y',z',w'\rangle}\frac{\min\{1,|x'z'|\}\min\{1,|y'w'|\}}{\min\{1,|x'y'|\}\min\{1,|z'w'|\}},
\end{equation}
with $C = C(\delta,K,\e',M)$.

Since $f$ is an $H$-biLipschitz map between geodesic $\delta$-hyperbolic spaces, by \cite[Theorem 4.4.1]{BS07} we have that $f$ roughly quasi-preserves the cross difference in the following sense: there are constants $C_{0} \geq 1$ and $c_{0} \geq 0$ depending only on $H$ and $\delta$ such that the cross-difference of the points $x,y,z,w$ compared to the image points $x',y',z',w'$ satisfies
\begin{equation}\label{PQ}
C_{0}^{-1}\langle x,y,z,w\rangle - c_{0} \leq \langle x',y',z',w'\rangle \leq C_{0}\langle x,y,z,w\rangle + c_{0},
\end{equation}
whenever $\langle x,y,z,w\rangle \geq 0$. We set $\beta = \max\{\frac{\e'}{\e},\frac{\e}{\e'}\}C_{0} \geq 1$. Then in the case $\langle x,y,z,w\rangle \geq 0$ we have
\begin{equation}\label{image cross difference inequality}
C^{-1} (e^{\e \langle x,y,z,w\rangle})^{\beta^{-1}} \leq e^{\e'\langle x',y',z',w'\rangle} \leq C (e^{\e \langle x,y,z,w\rangle})^{\beta},
\end{equation}
with $C = C(\delta,K,\e,\e',M,H)$. 

We set 
\[
G(x,y,z,w) = \frac{\min\{1,|xz|\}\min\{1,|yw|\}}{\min\{1,|xy|\}\min\{1,|zw|\}}.
\]
Let us first assume in addition to $\langle x,y,z,w \rangle \geq 0$ that we also have $G(x,y,z,w) \geq 1$. Then we have
\[
G(x,y,z,w)^{\beta^{-1}}\leq G(x,y,z,w) \leq G(x,y,z,w)^{\beta},
\]
which implies by inequality \eqref{image cross difference inequality} and the comparison \eqref{first comparison cross} that
\begin{equation}\label{improved cross}
C^{-1}[x,y,z,w]_{d_{\e,b}}^{\beta^{-1}} \leq e^{\e'\langle x',y',z',w'\rangle}\frac{\min\{1,|xz|\}\min\{1,|yw|\}}{\min\{1,|xy|\}\min\{1,|zw|\}} \leq  C[x,y,z,w]_{d_{\e,b}}^{\beta},
\end{equation}
with $C = C(\delta,K,\e,\e',M,H)$. Since $f$ is $H$-biLipschitz we have 
\begin{equation}\label{four bilip}
\frac{\min\{1,|x'z'|\}\min\{1,|y'w'|\}}{\min\{1,|x'y'|\}\min\{1,|z'w'|\}} \asymp_{H^{4}} \frac{\min\{1,|xz|\}\min\{1,|yw|\}}{\min\{1,|xy|\}\min\{1,|zw|\}},
\end{equation}
which when combined with inequalities \eqref{improved cross} and \eqref{second comparison cross} gives us
\begin{equation}\label{concluding cross}
C^{-1}[x,y,z,w]_{d_{\e,b}}^{\beta^{-1}} \leq [x',y',z',w']_{d_{\e',b'}} \leq C[x,y,z,w]_{d_{\e,b}}^{\beta},
\end{equation}
with $C$ depending only on the given data.

Let's now assume instead that $\langle x,y,z,w\rangle \geq 0$ and $G(x,y,z,w) \leq 1$. Then it must be the case that either $|xz| \leq 1$ or $|yw| \leq 1$. Since we have the equality $\langle x,y,z,w \rangle = \langle y,x,w,z \rangle$ when we interchange the roles of the pairs $(x,z)$ and $(y,w)$,  we can assume without loss of generality that $|xz| \leq 1$. Then by the triangle inequality we have $|zw| \doteq_{1} |xw|$ and therefore
\[
\langle x,y,z,w\rangle \doteq_{1} \frac{1}{2}(|yw|-|xy|-|xw|) \leq 0. 
\]
But since we've also assumed that $\langle x,y,z,w \rangle \geq 0$, it then follows that $0 \leq \langle x,y,z,w \rangle \leq 1$. Applying inequality \eqref{PQ} then gives that $|\langle x',y',z',w'\rangle | \leq c$, with $c = c(\delta,H)$ depending only on $\delta$ and $H$. Thus in this case, by combining the comparison \eqref{four bilip} with the comparisons \eqref{first comparison cross} and \eqref{second comparison cross}, we conclude simply that 
\begin{equation}\label{concluding second cross}
[x',y',z',w']_{d_{\e',b'}} \asymp_{C'} [x,y,z,w]_{d_{\e,b}},
\end{equation}
with $C'$ depending only on the given data. 

Now suppose that $\langle x,y,z,w\rangle \leq 0$. Then 
\[
\langle w,y,z,x\rangle = -\langle x,y,z,w\rangle \geq 0,
\]
and $G(w,y,z,x) = G(x,y,z,w)^{-1}$. We thus conclude that inequality \eqref{concluding cross} holds if $G(x,y,z,w) \leq 1$ and inequality \eqref{concluding second cross} holds if $G(x,y,z,w) \geq 1$, with $x$ and $w$ swapping places in these inequalities. But since $[w,y,z,x]_{d_{\e,b}} = [x,y,z,w]_{d_{\e,b}}^{-1}$ and $[w',y',z',x']_{d_{\e',b'}} = [x',y',z',w']_{d_{\e',b'}}^{-1}$, it follows from this that for $\langle x,y,z,w\rangle \leq 0$ and  $G(x,y,z,w) \leq 1$ we have
\begin{equation}\label{inverted concluding cross}
C^{-1}[x,y,z,w]_{d_{\e,b}}^{\beta} \leq [x',y',z',w']_{d_{\e',b'}} \leq C[x,y,z,w]_{d_{\e,b}}^{\beta^{-1}},
\end{equation}
and if $\langle x,y,z,w\rangle \leq 0$ and  $G(x,y,z,w) \geq 1$ then
\begin{equation}\label{inverting concluding second cross}
[x',y',z',w']_{d_{\e',b'}} \asymp_{C'} [x,y,z,w]_{d_{\e,b}},
\end{equation}
with $C$, $C'$, and $\beta$ the same as in \eqref{concluding cross} and \eqref{concluding second cross}. Combining \eqref{concluding cross}, \eqref{concluding second cross}, \eqref{inverted concluding cross}, and \eqref{inverting concluding second cross}, and noting that $\beta \geq 1$, we conclude that $f$ is $\eta$-quasim\"obius with $\eta(t) = C\max\{t^{\beta},t^{\beta^{-1}}\}$, $C = C(\delta,K,\e,\e',M,H)$. 
\end{proof}

\begin{rem}\label{specializations}
There are some special cases in which the conclusions of Proposition \ref{quasimobius uniform} can be improved. 

\begin{enumerate}
\item If $X = X'$ and $f: X \rightarrow X$ is the identity map then we can take $C_{0} = 1$ and $c_{0} = 0$ in \eqref{PQ}. This implies in particular that the control function $\eta$ has the form $\eta(t) = C\max\{t^{\frac{\e'}{\e}},t^{\frac{\e}{\e'}}\}$. This proves Theorem \ref{quasimobius uniformizations}. 

\item If $\omega_{b} \in X$, $\omega_{b'} \in X'$ and $f(\omega_{b}) = \omega_{b'}$ then an adaptation of the argument of \cite[Proposition 4.15]{BHK} shows that the map $f$ in Proposition \ref{quasimobius uniform} is $\eta$-quasisymmetric with $\eta$ depending only on $\delta$, $K$, $\e$, $\e'$, $M$, and $H$. In the case $f(\omega_{b}) \neq \omega_{b'}$ this argument still shows that the map $f$ is $\eta$-quasisymmetric, but the control function $\eta$ depends additionally on the distance $|f(\omega_{b})\omega_{b'}|$. 

\end{enumerate}
\end{rem}

Our next goal is to prove Theorem \ref{transfer lip}. For this purpose we will first show that the notions of being $\p$-Lipschitz and $\p$-biLipschitz are closed under composition. This does not require any assumptions on the metric spaces in question. 

\begin{prop}\label{composition true}
Let $f_{1}: (\Omega,d) \rightarrow (\Omega',d')$ and $f_{2}: (\Omega',d') \rightarrow (\Omega'',d'')$ be homeomorphisms of incomplete metric spaces such that $f_{i}$ is $\p$-Lipschitz with data $(L_{i},\la_{i})$, $i = 1,2$. Then $f_{2} \circ f_{1}$ is $\p$-Lipschitz with data $(L,\la)$ where $L = L_{2}L_{1}$ and $\la = \min\{\la_{1},L_{1}^{-1}\la_{2}\}$.    

Similarly if $f_{i}$ is $\p$-biLipschitz with data $(L_{i},\la_{i})$, $i = 1,2$, then  $f_{2} \circ f_{1}$ is $\p$-biLipschitz with data $(L,\la)$ where $L = L_{2}L_{1}$ and $\la = \min\{L_{2}^{-1}\la_{1},L_{1}^{-1}\la_{2}\}$.  
\end{prop}

\begin{proof}
Let $\la = \min\{\la_{1},L_{1}^{-1}\la_{2}\}$. Then for $y,z \in B_{d}(x,\la d_{\Omega}(x))$ we have 
\[
\frac{d'(f_{1}(y),f_{1}(z))}{d'_{\Omega'}(f_{1}(x))} \leq L_{1}\frac{d(y,z)}{d_{\Omega}(x)} < \la_{2}. 
\]
Thus we can apply the inequality \eqref{controlled inequality} for $f_{2}$ to $f_{1}(y)$ and $f_{1}(z)$ to get for $y,z \in B_{d}(x,\la d_{\Omega}(x))$,
\[
\frac{d''(f_{2}(f_{1}(y)),f_{2}(f_{1}(z)))}{d''_{\Omega''}(f_{2}(f_{1}(x)))} \leq L_{2} \frac{d'(f_{1}(y),f_{1}(z))}{d'_{\Omega'}(f_{1}(x))} \leq L_{1}L_{2} \frac{d(y,z)}{d_{\Omega}(x)}.
\]
Thus inequality \eqref{controlled inequality} holds for $f_{2} \circ f_{1}$ on $B_{d}(x,\la d_{\Omega}(x))$ with constant $L = L_{2}L_{1}$. The claim regarding $\p$-biLipschitz maps follows from applying this argument to the composition $f_{1}^{-1} \circ f_{2}^{-1} = (f_{2} \circ f_{1})^{-1}$ and noting that in this case we have $L_{i}^{-1} \leq 1$, $i = 1,2$.  
\end{proof}

\begin{proof}[Proof of Theorem \ref{transfer lip}]
We let $Y = (\Omega,k)$ and $Y' = (\Omega',k')$ denote the quasihyperbolizations of $\Omega$ and $\Omega'$ respectively. We will use the notation $d(x) = d_{\Omega}(x)$ for $x \in \Omega$ and $d'(x) = d_{\Omega'}(x)$ for $x \in \Omega'$. We first assume that $f: \Omega \rightarrow \Omega'$ is $\p$-Lipschitz with data $(L,\la)$. Since $Y$ is geodesic it suffices to prove that the induced map $f:Y \rightarrow Y'$ is locally $H$-Lipschitz, i.e., that for any $x \in Y$ there is an open neighborhood $U_{x}$ of $x$ on which $f$ is $H$-Lipschitz. We will take $U_{x} = B_{d}(x,\la d(x))$. Then by inequalities \eqref{controlled inequality} and \eqref{hyperbolize comparison} as well as inequality \eqref{elementary}, for $y \in U_{x}$ we have
\begin{align*}
k'(f(x),f(y)) &\leq 4A^{2}\log\left(1+ \frac{d'(f(x),f(y))}{\min\{d'(x),d'(y)\}}\right) \\
&\leq 4A^{2}\log\left(1+ L\frac{d(x,y)}{\min\{d(x),d(y)\}}\right) \\
&\leq 4A^{2}L\log\left(1+ \frac{d(x,y)}{\min\{d(x),d(y)\}}\right) \\
&\leq 4A^{2}L k(x,y). 
\end{align*}
We conclude that $f$ is $H$-Lipschitz with $H = 4A^{2}L$. 

We now assume that the induced map $f:Y \rightarrow Y'$ is $H$-Lipschitz for some $H \geq 0$. If $\Omega$ is bounded then we let $\omega \in \Omega$ be such that $d(\omega) = \sup_{x \in \Omega} d(x)$ and set $b(x) = k(x,\omega)$, while if $\Omega$ is unbounded then we let $b$ be a Busemann function on $Y$ based at the point $\omega \in \p Y$ corresponding to the equivalence class of all quasihyperbolic geodesic rays $\gamma$ in $\Omega$ with $\l_{d}(\gamma) = \infty$. We define $b' \in \hat{\mathcal{B}}(Y')$ similarly. Then $Y$ and $Y'$ are each $K$-roughly starlike from the basepoints of $b$ and $b'$ respectively with $K = K(A)$ by Proposition \ref{uniform starlike}.

By Theorem \ref{Gehring-Hayman} we can find $\e = \e(A)$ such that the densities $\rho_{\e,b}$ and $\rho_{\e,b'}$ are GH-densities on $Y$ and $Y'$ respectively with constant $M = 20$. Let $Y_{\e,b}$ and $Y_{\e,b'}'$ be the respective uniformizations of $Y$ and $Y'$. The first part of the proof of Proposition \ref{quasimobius uniform} implies that the induced map $f: Y_{\e,b} \rightarrow Y_{\e,b'}'$ is $\p$-Lipschitz with data $(L,\la)$, where $L$ and $\la$  depend only on $A$ and $H$ (since $\delta = \delta(A)$, $K = K(A)$, $\e = \e(A)$, and $M = 20$).  By Proposition \ref{uniformize quasi} the identity maps $\Omega \rightarrow Y_{\e,b}$ and $\Omega' \rightarrow Y_{\e,b'}'$ are each $\p$-biLipschitz with data $(L_{0},\la_{0})$ depending only on $A$ since $\e = \e(A)$ and $M = 20$. Since inverses of $\p$-biLipschitz maps are $\p$-biLipschitz and compositions of $\p$-Lipschitz maps are $\p$-Lipschitz (by Proposition \ref{composition true}) we conclude that $f: \Omega \rightarrow \Omega'$ is $\p$-Lipschitz with data $(L,\la)$ depending only on $A$ and $H$. 
\end{proof}

We can also prove Theorem \ref{homothety to mobius}. We use the same notation as in the proof of Theorem \ref{transfer lip}.

\begin{proof}[Proof of Theorem \ref{homothety to mobius}]
By Theorem \ref{transfer lip} there is an $H = H(A,L)$ such that the induced map $f: Y \rightarrow Y'$ is $H$-biLipschitz. By Proposition \ref{quasimobius uniform} the induced map $f: Y_{\e,b} \rightarrow Y_{\e,b'}'$ is $\eta_{0}$-quasim\"obius with $\eta_{0}$ depending only on $A$ and $H$ (and therefore only on $A$ and $L$). By Proposition \ref{uniformize quasi} the identity maps $\Omega \rightarrow Y_{\e,b}$ and $\Omega' \rightarrow Y_{\e,b'}'$ are $\eta_{1}$-quasisymmetric with $\eta_{1}$ depending only on $A$. Since quasisymmetric maps are quasim\"obius with quantitative control over the control function \cite[Theorem 3.2]{V85} and both composititions and inverses of quasim\"obius maps are quasim\"obius, it follows that $f: \Omega \rightarrow \Omega'$ is $\eta$-quasim\"obius with $\eta$ depending only on $A$ and $L$. 
\end{proof}

Our final task is to prove Theorem \ref{inversion theorem}. For all of the claims below we fix a proper geodesic $\delta$-hyperbolic space $X$ that is $K$-roughly starlike from some $z \in X$ and $\omega \in \p X$.  We fix a Busemann function $b$ based at $\omega$ that is chosen such that $b(z) = 0$ (this can always be done by adding a constant to $b$ if necessary).  We let $\e > 0$ be given such that the densities $\rho_{\e,z}$ and $\rho_{\e,b}$ are GH-densities for $X$ with constant $M$. By \cite[Theorem 1.4]{Bu20} the metric spaces $X_{\e,z}$ and $X_{\e,b}$ are each $A$-uniform with $A = A(\delta,K,\e,M)$. We will consider the cases of inversion and sphericalization separately. 

We start with the case of inversion. We recall that $X_{\e,z}^{\omega}$ denotes the inversion of $X_{\e,z}$ based at the point $\omega \in \p X_{\e,z}$ defined prior to Theorem \ref{inversion theorem}. We will need the following lemma. 

\begin{lem}\label{bounded ratio}
There is a constant $C = C(\delta,K,\e,M) \geq 1$ such that $\diam \, X_{\e,z} \asymp_{C} 1$. Furthermore we have
\[
\diam \, X_{\e,z} \leq C\diam \, \p X_{\e,z}. 
\]
\end{lem}

\begin{proof}
By a simple calculation \cite[(4.3)]{BHK} we have $\diam \, X_{\e,z} \leq 2\e^{-1}$. On the other hand, since $X$ is $K$-roughly starlike from $\omega$ we can find a geodesic line $\gamma: \R \rightarrow X$ starting from $\omega$ that is parametrized such that $|z\gamma(0)| \leq K$. Let $\xi \in \p X$ be the other endpoint of $\gamma$. Then, since $X_{\e,z}$ is $A$-uniform with $A = A(\delta,K,\e,M)$, by Lemma  \ref{other height} and Lemma \ref{compute distance} we have
\[
d_{\e,z}(\xi,\omega) \asymp_{C} \sup_{t \in \R} d_{\e,z}(\gamma(t))  \asymp_{C} \sup_{t \in \R} \rho_{\e,z}(\gamma(t)).
\]
with $C = C(\delta,K,\e,M)$. By the Harnack inequality \eqref{Harnack} we conclude that
\[
d_{\e,z}(\xi,\omega) \geq C^{-1}\rho_{\e,z}(\gamma(0)) \asymp_{e^{\e K}} \rho_{\e,z}(z) = 1. 
\]
Thus $d_{\e,z}(\xi,\omega) \geq C^{-1}$, where $C = C(\delta,K,\e,M)$. By continuity it follows that 
\[
d_{\e,z}(\gamma(t),\gamma(-t)) \geq \frac{1}{2}C^{-1}
\]
for all sufficiently large $t$. The lower bound $\diam \, X_{\e,z} \geq \frac{1}{2}C^{-1}$ follows. This gives the first claim of the lemma. The second claim of the lemma follows since 
\[
C^{-1} \leq d_{\e,z}(\xi,\omega) \leq \diam \, \p X_{\e,z}.
\]  
\end{proof}

\begin{prop}\label{prop inversion}
The metric space $X_{\e,z}^{\omega}$ is $A'$-uniform with $A' = A'(\delta,K,\e,M)$. The identity map $X_{\e,b} \rightarrow X_{\e,z}^{\omega}$ is $\p$-biLipschitz with data $(L,\la)$ and $\eta$-quasisymmetric with $L$, $\la$, and $\eta$ depending only on $\delta$, $K$, $\e$, and $M$.
\end{prop}

\begin{proof}
By \cite[Theorem 5.1(a)]{BHX08} the metric space $X_{\e,z}^{\omega}$ is $A'$-uniform with $A'$ depending only on the uniformity parameter$ A$ of $X_{\e,z}$. The identity map $X_{\e,z} \rightarrow X_{\e,b}$ is $\p$-biLipschitz with data depending only on $\delta$, $K$, $\e$, and $M$ by Proposition \ref{quasimobius uniform}, hence is $H$-biLipschitz in the quasihyperbolic metrics on these spaces with $H=H(\delta,K,\e,M)$ by Theorem \ref{transfer lip}. By \cite[Theorem 4.7]{BHX08} the identity map $X_{\e,z}\rightarrow X_{\e,z}^{\omega}$  is $H'$-biLipschitz in the quasihyperbolic metrics  with $H'$ depending only on the uniformity parameter $A = A(\delta,K,\e,M)$ of $X_{\e,z}$ and an upper bound on the ratio $\phi(X_{\e,z})$ (see \eqref{tightness}) of the diameter of $X_{\e,z}$ to the diameter of $\p X_{\e,z}$. By Lemma \ref{bounded ratio} this ratio is bounded above in terms of $\delta,K,\e,M$, so it follows that $H' = H'(\delta,K,\e,M)$. 

We thus conclude that there is a constant $H'' = H''(\delta,K,\e,M)$ such that the identity map $X_{\e,b} \rightarrow X_{\e,z}^{\omega}$ is $H''$-biLipschitz in the quasihyperbolic metrics on these uniform spaces. By Theorems \ref{homothety to mobius} and \ref{transfer lip} it then follows that this identity map is $\p$-biLipschitz with data $(L,\la)$ and $\eta$-quasim\"obius with $L$, $\la$, and $\eta$ depending only on $\delta$, $K$, $\e$, and $M$. To complete the proof we will show that the identity map $X_{\e,b} \rightarrow X_{\e,z}^{\omega}$ is actually $\eta$-quasisymmetric with the same control function $\eta$. Since both $X_{\e,b}$ and $X_{\e,z}^{\omega}$ are unbounded, it suffices by \cite[Theorem 3.10]{V85} to show that if $\{x_{n}\}$ is any sequence in $X$ then $d_{\e,b}(z,x_{n}) \rightarrow \infty$ if and only if $d^{\omega}_{\e,z}(z,x_{n}) \rightarrow \infty$. 

We recall from the definition \eqref{define inversion} of the inversion of $X_{\e,z}$ based at $\omega$ that we have for any $x,y \in X_{\e,z}$, 
\[
d^{\omega}_{\e,z}(x,y) \asymp_{4} \frac{d_{\e,z}(x,y)}{d_{\e,z}(x,\omega)d_{\e,z}(y,\omega)}.
\]
From this comparison we see that $d^{\omega}_{\e,z}(z,x_{n}) \rightarrow \infty$ if and only if $d_{\e,z}(x_{n},\omega) \rightarrow 0$. The comparison \eqref{estimate both} shows that this happens if and only if $(x_{n}|\omega)_{z} \rightarrow \infty$. By Lemma \ref{go to infinity} this occurs if and only if $d_{\e,b}(z,x_{n}) \rightarrow \infty$. This completes the proof. 
\end{proof}

In our final proposition we handle the sphericalization side of Theorem \ref{inversion theorem}. Theorem \ref{inversion theorem} then follows by combining Propositions \ref{prop inversion} and \ref{prop sphericalization}. We recall that $\hat{X}_{\e,b}^{z}$ denotes the sphericalization of $X_{\e,b}$ based at the point $z \in X_{\e,b}$ defined prior to Theorem \ref{inversion theorem}. 

\begin{prop}\label{prop sphericalization}
The metric space $\hat{X}_{\e,b}^{z}$ is $A'$-uniform with $A' = A'(\delta,K,\e,M)$. The identity map $X_{\e,z} \rightarrow \hat{X}_{\e,b}^{z}$ is $\p$-biLipschitz with data $(L,\la)$ and $\eta$-quasisymmetric with $L$, $\la$, and $\eta$ depending only on $\delta$, $K$, $\e$, and $M$.
\end{prop}

\begin{proof}
 By \cite[Theorem 5.5(a)]{BHX08} the metric space $\hat{X}_{\e,b}^{z}$ is $A'$-uniform with $A'$ depending only on $A$ and an upper bound on $d_{\e,b}(z)$. Since by Lemma \ref{compute distance},
\[
d_{\e,b}(z) \asymp_{C} e^{-\e b(z)} = 1,
\]
with $C = C(\delta,K,\e,M)$, we conclude that $\hat{X}_{\e,b}^{z}$ is $A'$-uniform with $A' = A'(\delta,K,\e,M)$. The identity map $X_{\e,b} \rightarrow X_{\e,z}$ is $\p$-biLipschitz with data depending only on $\delta$, $K$, $\e$, and $M$ by Proposition \ref{quasimobius uniform}, hence is $H$-biLipschitz in the quasihyperbolic metrics on these spaces with $H=H(\delta,K,\e,M)$ by Theorem \ref{transfer lip}. 

Viewing sphericalization as a special case of inversion as we did prior to the statement of Theorem \ref{inversion theorem}, \cite[Theorem 4.12]{BHX08} shows that the identity map $X_{\e,b} \rightarrow \hat{X}_{\e,b}^{z}$ is $H'$-biLipschitz in the quasihyperbolic metrics with $H'$ depending only on the uniformity parameter $A=A(\delta,K,\e,M)$ of $X_{\e,b}$; we emphasize for the purposes of applying the theorem in the reference that the augmented metric space $X_{\e,b} \cup_{z \sim 0} [0,1]$ to which we are applying it is unbounded. We thus conclude as in the case of Proposition \ref{prop inversion} that the identity map $X_{\e,z} \rightarrow \hat{X}_{\e,b}^{z}$ is $H''$-biLipschitz in the hyperbolic metrics for $H'' = H''(\delta,K,\e,M)$. Therefore by combining Theorems \ref{homothety to mobius} and \ref{transfer lip} we conclude that the identity map $X_{\e,z} \rightarrow \hat{X}_{\e,b}^{z}$ is $\p$-biLipschitz with data $(L,\la)$ and $\eta$-quasim\"obius with $L$, $\la$, and $\eta$ depending only on $\delta$, $K$, $\e$, and $M$.

Finally we must upgrade the identity map $X_{\e,z} \rightarrow \hat{X}_{\e,b}^{z}$ from being $\eta$-quasim\"obius to being $\eta'$-quasisymmetric with $\eta'$ depending only on $\delta$, $K$, $\e$, and $M$. We will employ a criterion due to V\"ais\"al\"a \cite[Theorem 3.12]{V85} which shows that it suffices to find a constant $\kappa = \kappa(\delta,K,\e,M) >0$ and a triple of points $x_{i} \in X_{\e,z}$, $i =1,2,3$, such that we have
\[
d_{\e,z}(x_{i},x_{j}) \geq \kappa \, \diam\, X_{\e,z},
\]
for $i \neq j$, and 
\[
\hat{d}_{\e,b}^{z}(x_{i},x_{j}) \geq \kappa\,\diam\, \hat{X}_{\e,b}^{z},
\]
for $i \neq j$. Since $\diam\, X_{\e,z} \asymp_{C} 1$ with $C = C(\delta,K,\e,M)$ by Lemma \ref{bounded ratio} and $\diam\, \hat{X}_{\e,b}^{z} \asymp_{4} 1$ by \cite[Section 3.B]{BHX08}, it suffices to produce a triple of points $x_{i} \in X$ and $\kappa > 0$ such that for $i \neq j$, 
\begin{equation}\label{sphericalization estimates}
\min\{d_{\e,z}(x_{i},x_{j}), \hat{d}_{\e,b}^{z}(x_{i},x_{j})\} \geq \kappa
\end{equation}
As in the proof of Lemma \ref{bounded ratio}, we let $\gamma: \R \rightarrow X$ be a geodesic line starting from $\omega$ such that $\dist(\gamma(0),z) \leq K$. We will set $x_{1} = \gamma(0)$, $x_{2} = \gamma(t)$, and $x_{3} = \gamma(-t)$ for a sufficiently large $t \geq 0$. 

We take $t \geq 1$ so that $|x_{i}x_{j}| \geq 1$ for $i \neq j$. The triangle inequality shows that we have $(x|y)_{z} \doteq_{K} (x|y)_{\gamma(0)}$ for any $x,y \in X$. The comparison \eqref{estimate both} then gives 
\[
d_{\e,z}(\gamma(0),\gamma(t)) \asymp_{C} e^{-\e (\gamma(0)|\gamma(t))_{\gamma(0)}} = 1,
\]
with $C = C(\delta,K,\e,M)$. Similar calculations give that $d_{\e,z}(\gamma(0),\gamma(-t)) \asymp_{C} 1$ and $d_{\e,z}(\gamma(t),\gamma(-t)) \asymp_{C} 1$ with the same constant $C$. This takes care of the estimates corresponding to $d_{\e,z}$ in \eqref{sphericalization estimates}. 

The distance $\hat{d}_{\e,b}^{z}$ in the sphericalization has the form for $x,y \in X$, 
\begin{equation}\label{sphericalization formula}
\hat{d}_{\e,b}^{z}(x,y) \asymp_{4} \frac{d_{\e,b}(x,y)}{(1+d_{\e,b}(x,z))(1+d_{\e,b}(y,z))}.
\end{equation}
Since $b$ is $1$-Lipschitz we have $(x|z)_{b} \doteq_{K} (x|\gamma(0))_{b}$ for any $x \in X$. Since $\gamma$ is a geodesic starting from $\omega$ we have by \cite[Lemma 2.6]{Bu20} that for $t \in \R$, 
\[
b(\gamma(t)) \doteq_{c(\delta)} t+ b(\gamma(0)) \doteq_{K} t+ b(z) = t.
\]
 The comparison \eqref{estimate both} then yields $d_{\e,b}(\gamma(0),z) \asymp_{C} 1$, $d_{\e,b}(\gamma(t),z) \asymp_{C} e^{-\e t}$, and $d_{\e,b}(\gamma(-t),z) \asymp_{C} e^{\e t}$. By similar calculations we obtain that $d_{\e,b}(\gamma(0),\gamma(t)) \asymp_{C} 1$, $d_{\e,b}(\gamma(0),\gamma(-t)) \asymp_{C} e^{\e t}$, and $d_{\e,b}(\gamma(t),\gamma(-t)) \asymp_{C} e^{\e t}$. Plugging these comparisons into \eqref{sphericalization formula} gives the following three inequalities, all with constants $C = C(\delta,K,\e,M)$ (recall that $t \geq 1$), 
\[
\hat{d}_{\e,b}^{z}(\gamma(0),\gamma(t)) \geq \frac{C^{-1}}{(1+C)(1+Ce^{-\e t})} \geq \frac{C^{-1}}{(1+C)^{2}},
\]
and 
\[
\hat{d}_{\e,b}^{z}(\gamma(0),\gamma(-t)) \geq \frac{C^{-1}e^{\e t}}{(1+C)(1+Ce^{\e t})} \geq \frac{C^{-1}}{(1+C)^{2}},
\]
and finally 
\[
\hat{d}_{\e,b}^{z}(\gamma(t),\gamma(-t)) \geq \frac{C^{-1}e^{\e t}}{(1+Ce^{-\e t})(1+Ce^{\e t})} \geq \frac{C^{-1}}{(1+C)^{2}}.
\]
These three inequalities complete the proof of inequality \eqref{sphericalization estimates}. We conclude that the identity map $X_{\e,z} \rightarrow \hat{X}_{\e,b}^{z}$ is $\eta'$-quasisymmetric with $\eta'$ depending only on $\eta$ and $\kappa$, and therefore only on $\delta$, $K$, $\e$, and $M$. 
\end{proof}

\bibliographystyle{plain}
\bibliography{UniformInversion}

\end{document}